\documentclass[a4paper,10pt]{article}
\usepackage{amsmath, amssymb, latexsym, amsfonts, amsthm}
\usepackage{fullpage}
\usepackage{graphicx}
\usepackage{url}
\newtheorem{theorem}{Theorem}[section]
\newtheorem{conjecture}[theorem]{Conjecture}
\newtheorem{rtheorem}[theorem]{Reduction Theorem}
\newtheorem{lemma}[theorem]{Lemma}
\newtheorem{corollary}[theorem]{Corollary}
\newtheorem{problem}[theorem]{Problem}
\newtheorem{remark}[theorem]{Remark}
\newtheorem{proposition}[theorem]{Proposition}

\newtheorem{example}[theorem]{Example}
\newcommand{\sh}{\sqcup \mspace{-7.0mu} \sqcup}

\numberwithin{equation}{section}
\title{\bf On the support of the free Lie algebra: the Sch\"utzenberger problems \\}
\author{\bf Ioannis C. Michos \thanks{E-mail addresses: michos@ucy.ac.cy, \, yannis\_michos@yahoo.co.uk}\\ \\
\small{Department of Mathematics and Statistics, University of Cyprus}, \\
\small{P.O. Box 20537, CY-1678, Nicosia, Cyprus}}
\date{}
\begin{document}
\maketitle
\begin{abstract}

M.-P. Sch\"utzenberger asked to determine the support of the free Lie algebra ${\mathcal L}_{{\mathbb Z}_{m}}(A)$ on a finite alphabet $A$
over the ring ${\mathbb Z}_{m}$ of integers $\bmod \, m$ and all pairs of {\em twin} and {\em anti-twin} words,
\textit{i.e.}, words that appear with equal (resp. opposite) coefficients in each Lie polynomial.
We characterize the complement of the support of ${\mathcal L}_{{\mathbb Z}_{m}}(A)$ in $A^{*}$ as the set of words $w$ such that 
$m$ divides all the coefficients appearing in the monomials of $l^{*}(w)$, where $l^{*}$ is the adjoint endomorphism of the left normed Lie
bracketing $l$ of the free Lie ring. 
This can be rephrased, for words of length $n$, in terms of the action of the left normed multi-linear Lie bracketing $l_{n}$ of
${\mathcal L}_{{\mathbb Z}_{m}}(A)$ - viewed as an element of the group ring of the symmetric group on $n$ letters - on $\lambda$-tabloids,
where $\lambda$ is a partition of $n$.
Calculating $l^{*}(w)$ via all factors of $w$ of fixed length and the shuffle product,
we recover the result of Duchamp and Thibon $(1989)$ for the support of the free Lie ring in a much more natural way.
We conjecture that two words $u$ and $v$ of common length $n > 1$ which lie in the support of the free Lie ring are twin
(resp. anti-twin) if and only if either $u = v$ or $n$ is odd and $u = \tilde{v}$ (resp. $n$ is even and $u = \tilde{v}$), where $\tilde{v}$ denotes the reversal of the word $v$, and we prove that it suffices to show this only in the case where $|A| = 2$.
Representing a word $w$ in two letters by the subset $I$ of $[n] = \{1, 2, \ldots , n \}$ consisting of the positions that one of the letters
occurs in $w$, the computation of $l^{*}(w)$ leads us to the notion of the {\em Pascal descent polynomial} $p_{n}(I)$, a particular commutative multi-linear polynomial which is equal to a signed binomial coefficient when $|I| = 1$. We provide a recursion formula for it and show that
if $m \nmid \sum_{i \in I} (-1)^{i-1} {{n-1} \choose {i-1}}$ then $w$ lies
in the support of ${\mathcal L}_{{\mathbb Z}_{m}}(A)$.

\medskip

{\bf Keywords:} {Free Lie algebras, Pascal triangle $\bmod \, m$, shuffle product, set partitions, $\lambda$-tabloids}
\end{abstract}

\section{Introduction}\label{intro}

Let $A$ be a finite alphabet, $A^{*}$ be the {\em free monoid} on $A$ and $A^{+} = A^{*} \setminus \{ \epsilon \}$ be the {\em free semigroup}
on $A$, with $\epsilon$ denoting the empty word. For a word $w \in A^{*}$ let $|w|$ denote its length,
$|w|_{\!a}$ denote the number of occurrences of the letter $a \in A$ in $w$ and let $alph(w)$ be
the set of letters actually occurring in $w$.

Let $K$ be a commutative ring with unity. For most of our purposes $K = {\mathbb Z}_{m}$, the ring ${\mathbb Z}/(m)$ of integers $\bmod \, m$
for a non-negative integer $m$. Let $K {\langle A \rangle}$ be the {\em free associative algebra} on $A$ over $K$.
Its elements are the polynomials on non-commuting variables from $A$ and coefficients from $K$.
Each polynomial $P \in K {\langle A \rangle}$ is written in the form $P = \sum_{w \in A^{*}} \, (P, w) \, w$, where
$(P, w)$ denotes the coefficient of the word $w$ in $P$.
Given two polynomials $P, Q \in K {\langle A \rangle}$, their Lie product is the Lie bracket
$[P,Q] = PQ - QP$. In this way $K {\langle A \rangle}$ is given a Lie structure. The {\em free Lie algebra}
${\mathcal L}_{K}(A)$ on $A$ over $K$ is then equal to the Lie subalgebra of $K {\langle A \rangle}$ generated by $A$.
When $K$ is the ring of rational integers $\mathbb Z$, ${\mathcal L}_{K}(A)$ is also known as the {\em free Lie ring}.
A {\em Lie monomial} is an element of ${\mathcal L}_{K}(A)$ formed by Lie products of the elements $a \in A$.
A {\em Lie polynomial} is a linear combination of Lie monomials, \textit{i.e.}, an arbitrary element of ${\mathcal L}_{K}(A)$.
The {\em support} of ${\mathcal L}_{K}(A)$ is the subset of ${A}^{*}$ consisting of those
words that appear (with a nonzero coefficient) in some Lie polynomial.
A pair of words $u, \, v$ is called {\em twin} (respectively {\em anti-twin}) if both words
appear with equal (respectively opposite) coefficients in each Lie polynomial over $K$.

M.-P. Sch\"utzenberger had posed the following problems (private communication with G. Duchamp):

\begin{problem}\label{suppfreeLieZ}
Determine the support of the free Lie ring ${\mathcal L}_{\mathbb Z}(A)$.
\end{problem}

\begin{problem}\label{suppfreeLieZm}
Determine the support of ${\mathcal L}_{{\mathbb Z}_{m}}(A)$, for $m > 1$.
\end{problem}

\begin{problem}\label{twinantiZ}
Determine all the twin and anti-twin pairs of words with respect to ${\mathcal L}_{\mathbb Z}(A)$.
\end{problem}

\begin{problem}\label{twinantiZm}
Determine all the twin and anti-twin pairs of words with respect to ${\mathcal L}_{{\mathbb Z}_{m}}(A)$, for $m > 1$.
\end{problem}

In view of these problems Sch\"utzenberger considered, for each word $w \in {A}^{*}$, the smallest non-negative integer - which we denote by
$c(w)$ - that appears as a coefficient of $w$ in some Lie polynomial over $\mathbb Z$. For each non-negative integer $m$ he also defined
and tried to characterize the language $L_{m}$ of all words with $c(w) = m$; considering, in particular, the cases $m = 0$ and
$m = 1$ (see \cite[\S 1.6.1]{Reut}).

For $m = 0$ the language $L_{0}$ is clearly equal to the complement of the support of the free Lie ring
${\mathcal L}_{\mathbb Z}(A)$ in $A^{*}$, since a word $w$ does not appear in any Lie polynomial over
$\mathbb Z$ if and only if $c(w) = 0$.
Duchamp and Thibon gave a complete answer to Problem \ref{suppfreeLieZ} in \cite{Duch + Thib} and proved that
$L_{0}$ consists of all words $w$ which are either a power $a^{n}$ of a letter $a$,
with exponent $n > 1$, or a {\em palindrome} (\textit{i.e.}, a word $u$ equal to its {\em reversal}, denoted by
$\tilde{u}$) of even length.
The non-trivial part of their work was to show that each word not of the previous form lies in the support of
${\mathcal L}_{\mathbb Z}(A)$ and this was achieved by a construction of an ad hoc family of Lie polynomials.
This result was extended in \cite{Duch + Laug + Luqu} - under certain conditions - to {\em traces},
\textit{i.e.}, partially commutative words (see \cite{Diek + Rozen} for an exposition of trace theory) instead of
noncommutative ones, and the corresponding free partially commutative Lie algebra
(also known as {\em graph Lie algebra}).

For $m = 1$ all {\em Lyndon words} on $A$ (for more on this subject see e.g., \cite[\S 5.1 and \S 5.3]{Loth}) lie in $L_{1}$ since the
element $P_w$ of the {\em Lyndon basis} of ${\mathcal L}_{\mathbb Z}(A)$ that corresponds to the {\em standard factorization} of a given
Lyndon word $w$ is equal to $w$ plus a linear combination of greater words - with respect to the {\em lexicographic ordering} in $A^{+}$
- of the same length as $w$ \cite[Lemma 5.3.2]{Loth}. On the other hand, there exist non Lyndon words which also lie in $L_{1}$.
For example, one can check that the word $a^2b^2a$ - which is clearly non Lyndon as it starts and ends with the same letter - appears with
coefficient equal to $-1$ in the Lie monomial $P_{a^3b^2} \: = \: [\, a, \, [\, a, \, [\, [\, a, \, b \,], \, b \,] \,] \, ]$ and
therefore $(- P_{a^3b^2}, \, a^2b^2a) = 1$.

\smallskip

In Section \ref{prelim} we relate Problems \ref{suppfreeLieZ} up to \ref{twinantiZm} with the notion of the adjoint endomorphism
$l^{*}$ of the left normed Lie bracketing $l$ of the free Lie algebra ${\mathcal L}_{K}(A)$ over $K$.
Our starting point is the simple idea that a word $w$ does not lie in the support of ${\mathcal L}_{K}(A)$ if and only if $l^{*}(w) = 0$ and
a pair $u, \, v$ of words is twin (respectively anti-twin) if and only if $l^{*}(u) = l^{*}(v)$ (respectively $l^{*}(u) = - \, l^{*}(v)$).
We also show that $c(w)$ is either zero or is equal to the greatest common divisor of the coefficients of the
monomials appearing in $l^{*}(w)$, for the left normed Lie bracketing $l$ of the free Lie ring.
It turns out that it is also equal to the greatest common divisor of the coefficients in the expression of $l^{*}(w)$ as a linear combination of
the images of the Lyndon words of length $|w|$ under $l^{*}$.

Considering the natural projection from $\mathbb Z$ onto ${\mathbb Z}_{m}$ for $m \neq 1$, we show that the complement of
the support of ${\mathcal L}_{{\mathbb Z}_{m}}(A)$ is identified with the language ${\overline{L}}_{m}$
of all words $w$ with $m \, | \, c(w)$.
For Problem \ref{twinantiZ} we conjecture that two words $u, v$ of common length $n$ that do not lie in the support of the
free Lie ring, \textit{i.e.}, they are not $n$-th powers of a letter with $n > 1$ or palindromes of even length,
are twin if either $u = v$ or $n$ is odd and $u = \tilde{v}$ and are anti-twin if
$n$ is even and $u = \tilde{v}$. We also show that it suffices to prove this over an alphabet of two letters.

\smallskip

In Section \ref{l*shuffle} we calculate the polynomial $l^{*}(w)$ recursively in terms of
all factors $u$ of fixed length $r \geq 1$ of $w$ and the {\em shuffle product} of words (see \cite[\S 1.4]{Reut} for a definition) as
\begin{equation*}
l^{*}(w) \: = \: \sum_{{w = sut} \atop {|u|=r}} \, l^{*}(u) \, (-1)^{|s|} \, \{ \tilde {s} \, \sh \, \,  t \}
\end{equation*}
and use this to recover naturally the result of Duchamp and Thibon in \cite{Duch + Thib}.
Furthermore, if $w$ lies in the kernel of $l^{*}$ over $K$ we show that $|alph(w)| \leq \lceil |w|/2 \rceil$.
Applying this for $K = {\mathbb Z}_{m}$ for all $m \neq 1$ we obtain, as a corollary, the fact that all words $w$ with
$|alph(w)| > \lceil |w|/2 \rceil$ have $c(w) = 1$ and therefore lie in $L_{1}$, just as Lyndon words do.

\smallskip

In Section \ref{lnsymm} Problems \ref{suppfreeLieZ} up to \ref{twinantiZm} boil down to particular combinatorial questions
on the group ring $K {\mathfrak S}_{n}$ of the symmetric group ${\mathfrak S}_{n}$ on $n$ letters.
Let $[n]$ denote the set $\{ 1, 2, \ldots , n \}$ and fix an ordered sub-alphabet $B = \{ a_{1}, a_{2}, \ldots a_{r} \}$ of $A$.
The main idea is to view a word $w$ of length $n$ on $B$ as an ordered set partition of $[n]$ denoted by
$\{ w \} = (I_{1}(w), I_{2}(w), \ldots , I_{r}(w))$, where for each $k \in [r]$ the set $I_{k}(w)$ consists of the positions of $[n]$
in which the letter $a_k$ occurs in $w$. If $\lambda = ({\lambda}_{1}, {\lambda}_{2}, \ldots , {\lambda}_{r})$ is the multi-degree of $w$
then $\{ w \}$ is just a $\lambda$-tabloid, where $\lambda$ may, without loss of generality, assumed to be an integer partition of $n$.
The role of the reversal $\tilde{w}$ of a word $w$ is played by the tabloid ${{\tau}_{n}} \cdot \{ w \}$, where ${\tau}_{n}$ is the involution
$\prod_{i=1}^{k} \, (i, \, n-i+1)$ of ${\mathfrak S}_{n}$ with $k = {\left\lfloor n/2 \right\rfloor}$.
Viewing each permutation as a word in $n$ distinct letters, the left normed  multi-linear Lie bracketing
$l_{n} = l(x_1 x_2 \cdots x_n)$ and its adjoint $l^{*}_{n} = l^{*}(x_1 x_2 \cdots x_n)$
can be viewed as elements of the group ring $K {\mathfrak S}_{n}$.
The right permutation action of $l^{*}_{n}$ on words is then equivalent to the left natural action of $l_n$
on tabloids; in particular $w \cdot l^{*}_{n} \: = \: 0$ if and only if $l_{n} \cdot \{w\} \: = \: 0$.
In this way all previous results and problems translate to corresponding
problems on tabloids. In particular, Problem \ref{suppfreeLieZm} boils down to the problem of finding all $\lambda$-tabloids $t$
that satisfy the equation $l_{n} \cdot t = 0$ in the group ring ${\mathbb Z}_{m}{\mathfrak S}_{n}$.

\smallskip

One can say more for words where only two letters occur since an ordered partition with two parts is determined by
the subset $I = \{i_{1}, i_{2}, \ldots , i_{s} \}$ of $[n]$ that appears in its second row, and is denoted accordingly by $\overline{I}$.
We map $\overline{I}$ to the monomial $x_I = x_{i_{1}} x_{i_{2}} \cdots x_{i_{s}}$ in $n$ commuting variables
$x_1 , x_2 , \ldots , x_n$ and extending this by linearity the element $l_n \cdot \overline{I}$ can be viewed as
a multi-linear polynomial $h_{n}(I)$ of total degree $s$. In this way we can view all Sch\"utzenberger problems in a commutative algebra
setting.
We show that $h_{n}(I)$ is a multiple of the binomial $x_1 - x_2$.
The corresponding quotient $p_{n}(I)$ is called {\em Pascal descent polynomial} and is in many aspects a generalization of the notion of
the signed binomial coefficient. We study the polynomials $p_{n}(I)$ in Section \ref{pascal} and give a recursion formula for their calculation.
Problem \ref{suppfreeLieZm} is then equivalent to determining all subsets $I$ of $[n]$ such that $p_{n}(I) \equiv 0 \, (\bmod \, m )$.
The latter leads us to an explicitly stated sufficient condition  for a word $w$ to lie in the support
of the free Lie algebra ${\mathcal L}_{{\mathbb Z}_{m}}(A)$: namely $m \nmid N_{n}(I)$, where $n = |w|$,
$I = I(w)$ is the subset of $[n]$ consisting of the positions that one of the two letters occurs in $w$ and
$\displaystyle N_{n}(I) \, = \, \sum_{i \in I} (-1)^{i-1} {{n-1} \choose {i-1}}$. This actually means that the signed sum of the entries
appearing at the positions corresponding to the subset $I$ in the $n$-th row (starting to count from $n = 0$) of the {\em Pascal triangle}
$\bmod \, m$ has to be different from zero $\bmod \, m$.
Similar necessary conditions are obtained for twin and anti-twin words with respect to ${\mathcal L}_{{\mathbb Z}_{m}}(A)$.
Finally our conjecture for twin and anti-twin pairs in the free Lie ring is equivalent to showing that when
$p_{n}(I) \neq 0$ and $p_{n}(J) \neq 0$, then $p_{n}(I) = p_{n}(J)$ if and only if $I = J$ or $I = {{\tau}_{n}}(J)$
and $n$ is odd and $p_{n}(I) = - \, p_{n}(J)$ if and only if $I = {{\tau}_{n}}(J)$ and $n$ is even.

\bigskip


\section{Preliminary results}\label{prelim}

The set of all polynomials $K {\langle A \rangle}$ becomes a non-commutative associative algebra with the
usual {\em concatenation} product defined as
\begin{equation}
(PQ, \, w) \: = \: \sum_{w = uv} (P,u)(Q,v) \, ,
\end{equation}
and a commutative associative algebra with the {\em shuffle product} that is initially defined for words as
${\epsilon} \, \sh \: w  \: = \: w  \, \sh \, {\epsilon} \: = \: w$, if at least one of them is the empty word $\epsilon$, and recursively as
\begin{equation}
(au') \, \sh \, (bv') \: = \: a(u' \, \sh \, (bv')) \, + \, b((au') \, \sh \, v')  \, , \label{recdefshuffle}
\end{equation}
if $u = au'$ and $v = bv'$ with $a, b \in A$ and $u', v' \in A^{*}$ (see \cite[(1.4.2)]{Reut}).
Elements of the form $u \, \sh \, v$ with $u, v \in A^{+}$ are called {\em proper shuffles}.
The definition of the shuffle product is then extended linearly to the whole of $K {\langle A \rangle}$.

\smallskip

The {\em left normed Lie bracketing} of a word is the Lie polynomial defined recursively as
\begin{equation}
l({\epsilon}) \, = \, 0 \, , \qquad l(a) \, = \, a \, , \qquad \mbox{and} \qquad
l(ua) \, = \,  [l(u), \, a] \, ,
\end{equation}
for each $a \in A$ and $u \in {A}^{+}$.
One can extend $l$ linearly to $K {\langle A \rangle}$ and construct a linear map, denoted also by $l$, which maps
$K {\langle A \rangle}$ onto the free Lie algebra ${\mathcal L}_{K}(A)$, since the set
$\{ l(u) \: : \: u \in A^{*} \}$ is a well known linear generating set of ${\mathcal L}_{K}(A)$ (see e.g. \cite[\S 0.4.1]{Reut}).

Given two polynomials $P, Q \in K {\langle A \rangle}$ there is a canonical scalar product defined as
\begin{equation}
(P,Q) \: = \: \sum_{w \in {A}^{\, *}} (P,w)(Q,w) \, \label{scalarproduct}
\end{equation}
in the sense that it is the unique scalar product on $K {\langle A \rangle}$ for which ${A}^{*}$
is an orthonormal basis. The adjoint endomorphism $l^*$ of the left normed Lie bracketing $l$ is then defined
by the relation
\begin{equation}
({l^{*}}(u), v) \, = \, (l(v), u) \label{defl*}
\end{equation}
for any words $u$, $v$. The image of $l^{*}$ on a word of ${A}^{\, *}$ can also be effectively defined recursively
by the relations
\begin{equation}
{l^{*}}({\epsilon}) \, = \, 0 \, , \qquad {l^{*}}(a) \, = \, a \, , \qquad \mbox{and} \qquad
{l^{*}}(aub) \: = \: {l^{*}}(au)\,b - {l^{*}}(ub)\,a \, , \label{defrecl*}
\end{equation}
where $a, b \in A$ and $u \in {A}^*$ (cf. \cite[Problem 5.3.2]{Loth}). The proof goes by induction on the length of the given word,
just as in the case of the adjoint endomorphism of the right-normed Lie bracketing (discussed in detail in
\cite[pp. 32 - 33]{Reut}). The reason we choose to work with the left normed one is that there exist well known formulae for the left normed
multi-linear Lie bracketing of the free Lie algebra; this will be discussed later on in Section \ref{lnsymm}.

\smallskip
\begin{lemma}\label{l*reversal}
Let $\tilde{w}$ denote the reversal of the word $w$. Then
\[ {l^{*}}(\tilde{w}) \, = \, (-1)^{|w| + 1} {l^{*}}(w). \]
\end{lemma}

\begin{proof}
By the recursive formula (\ref{defrecl*}) and an easy induction on $|w|$.
\end{proof}


One can also extend $l^{*}$ linearly to the whole of $K {\langle A \rangle}$ and construct a linear endomorphism of $K {\langle A \rangle}$,
denoted also by $l^{*}$. What is of crucial importance for the Sch\"utzenberger problems is the kernel $\ker l^{*}$ of $l^{*}$.
Let ${{\mathcal L}_{K}(A)}^{\perp}$ denote the orthogonal complement of ${\mathcal L}_{K}(A)$ with respect to the scalar product
(\ref{scalarproduct}) in $K {\langle A \rangle}$. Then for an arbitrary commutative ring $K$ with unity the following two results hold.

\begin{lemma}\label{kerl*orth}
$\ker l^{*} = {{\mathcal L}_{K}(A)}^{\perp}$.
\end{lemma}

\begin{proof}
A polynomial $P \in \ker l^{*}$ if and only if
$\Big( \sum_{w \in A^{*}} \, (P, w) \, l^{*}(w), \: u \Big) = \sum_{w \in A^{*}} \, (P, w) (l^{*}(w), u) = 0$,
for each $u \in A^{*}$. By (\ref{defl*}) and (\ref{scalarproduct}) the latter means that
$(P, \, l(u)) = 0$, for each $u \in A^{*}$. Since $\{ l(u) \, : \, u \in A^{*} \}$ is a $K$-linear generating set for ${\mathcal L}_{K}(A)$, this is equivalent to having $(P, Q) = 0$ for each $Q \in {\mathcal L}_{K}(A)$, which means precisely that $P \in {{\mathcal L}_{K}(A)}^{\perp}$.
\end{proof}

\begin{lemma}\label{kerl*supp}
Let $u, v, w \in A^{*}$. Then
\begin{description}
\item[\em{(i)}]
A word $w$ does not lie in the support of ${\mathcal L}_{K}(A)$ if and only if $w \in \ker l^{*}$.
\item[\em{(ii)}]
A pair of words $(u, v)$ is twin (respectively anti-twin) with respect to ${\mathcal L}_{K}(A)$ if and only if
the binomial $u - v$ (respectively $u + v$) $\in \ker l^{*}$.
\end{description}
\end{lemma}

\begin{proof}
(i) It follows from Lemma \ref{kerl*orth} and the fact that a word $w$ does not lie in the support of ${\mathcal L}_{K}(A)$ if and only if
$(Q, w) = 0$, for each Lie polynomial $Q$.

(ii) Suppose that $u$ and $v$ are twin words. By definition this
means that $(Q, u) = (Q, v)$, for every $Q \in {\mathcal L}_{K}(A)$, \textit{i.e.}, the binomial
$u - v \in {{\mathcal L}_{K}(A)}^{\perp}$ and the result follows from Lemma \ref{kerl*orth}.
An analogous argument is used for anti-twin pairs and the binomial $u + v$.
\end{proof}

\begin{remark}\label{propershuffle}
By a result originally due to Ree \cite{Ree} it is known (see \cite[Theorem 3.1\,(iv)]{Reut} and cf. \cite[Ex. 5.3.4]{Loth}) that if
$K$ is assumed to be a commutative $\mathbb Q$-algebra, where $\mathbb Q$ denotes the field of rational numbers, a polynomial with zero constant
term lies in ${{\mathcal L}_{K}(A)}^{\perp}$ - and hence not in the support of ${\mathcal L}_{K}(A)$ in view of Lemmas \ref{kerl*orth} and
\ref{kerl*supp} - if and only if it is a $K$-linear combination of proper shuffles.
\end{remark}

\begin{proposition}\label{l*Lyndon}
Let $K$ be any commutative ring with unity and $l_{1}, \, l_{2}, \, \ldots , \, l_{k}$ be the Lyndon words of length $n$ on the alphabet $A$.
Then the set $\{ l^{*}(l_{1}), \, l^{*}(l_{2}), \, \ldots \, , l^{*}(l_{k}) \}$ is a $K$-basis of the image under $l^{*}$ of
the $n$-th homogeneous component of $K {\langle A \rangle}$.
\end{proposition}

\begin{proof}
Suppose that, without loss of generality, $l_{1} < l_{2} < \cdots < l_{k}$ with respect to the lexicographic ordering in $A^{+}$.
Consider also the corresponding Lyndon basis $\{ P_{l_{1}}, \, P_{l_{2}}, \, \ldots \, , P_{l_{k}} \}$ of ${\mathcal L}^{n}_{K}(A)$,
where each $P_{l_{i}}$ is written in the from $\displaystyle P_{l_{i}} = l_{i} + \sum_{t > l_{i}} c_{t} t$, for suitable $c_{t} \in K$
(see \cite[Lemma 5.3.2]{Loth}).

Let $w$ now be a given word of length $n$; we have to show that there exist unique coefficients
${\xi}_{1}, \, {\xi}_{2}, \, \ldots \, , \, {\xi}_{k} \in K$ such that
${\xi}_{1} \, l^{*}(l_{1}) \: + \: {\xi}_{2} \, l^{*}(l_{2}) \: + \cdots \: {\xi}_{k} \, l^{*}(l_{k}) \: = \: l^{*}(w)$.
This means that for all $u \in A^{n}$ we must have
$(l^{*}(l_{1}), u) \, {\xi}_{1} \: + \: (l^{*}(l_{2}), u) \, {\xi}_{2} \: + \: \cdots \: + (l^{*}(l_{k}), u) \, {\xi}_{k}
\: = \: (l^{*}(w), u)$, which, by (\ref{defl*}), is equivalent to
$(l(u), l_{1}) \, {\xi}_{1} \: + \: (l(u), l_{2}) \, {\xi}_{2} \: + \: \cdots \: (l(u), l_{k}) \, {\xi}_{k} \: = \:
(l(u), w)$. Since $\{ l(u) \: : \: \: u \in {A}^{n} \}$ and $\{ P_{l_{1}}, \, P_{l_{2}}, \, \ldots \, , P_{l_{k}} \}$ is a generating set
and a $K$-basis, respectively, for ${\mathcal L}^{n}_{K}(A)$ the latter is equivalent to the $k \times k$ linear system
$(P_{l_{i}}, l_{1}) \, {\xi}_{1} \: + \:  (P_{l_{i}}, l_{2}) \, {\xi}_{2} \: + \: \cdots \: + \: (P_{l_{i}}, l_{k}) \, {\xi}_{k}
\: = \:  (P_{l_{i}}, w), \quad i \in \{1, 2, \ldots , k \}$ in the unknowns ${\xi}_{1}, {\xi}_{2}, \ldots , {\xi}_{k}$.
Since $(P_{l_{i}}, l_{i}) = 1$ and $(P_{l_{i}}, l_{j}) = 0$ for $j < i$ this boils down to
\begin{equation}
{\xi}_{i} \: + \: (P_{l_{i}}, l_{i+1}) \, {\xi}_{i+1} \: + \cdots \: + \: (P_{l_{i}}, l_{k}) \, {\xi}_{k} \: = \:
(P_{l_{i}}, w), \quad i \in \{1, 2, \ldots , k \}. \label{Lyndcoeffw}
\end{equation}
It is now evident that the linear system \eqref{Lyndcoeffw} has a unique solution $({\xi}_{1}, {\xi}_{2}, \ldots , {\xi}_{k})$,
as required.
\end{proof}

\smallskip

For a given word $w$ of length $n$ Sch\"utzenberger considered the unique non-negative generator $c(w)$ of the ideal
$\{ (P,w) \, : \, P \in {\mathcal L}_{\mathbb Z}(A) \}$ of $\mathbb Z$ (see \cite[\S 1.6.1]{Reut}).
It is natural to ask how to calculate the number $c(w)$ and secondly try to find a Lie polynomial $Q_w$ of degree $n$ such that
$(Q_{w}, w) = c(w)$.
One would a priori search amongst all Lie polynomials of some basis of the $n$-th homogeneous component ${\mathcal L}^{n}_{\mathbb Z}(A)$.
It turns out that this can be done in terms of just one polynomial, which is not a Lie polynomial in general, namely the element
$l^{*}(w) \in {\mathbb Z}{\langle A \rangle}$.

\begin{theorem}\label{c(w)gcd}
Let $w$ be a word of length $n$ and $l^{*}$ be the adjoint endomorphism of the left normed Lie
bracketing $l$ of the free Lie ring on $A$. Then $c(w)$ is either zero or is equal to the greatest common divisor of the non-zero coefficients
that appear either in the monomials of the polynomial ${l^{*}}(w)$ or in its representation as a linear combination of the images
of the Lyndon words of length $n$ under $l^{*}$.
\end{theorem}

\begin{proof}
Clearly $c(w) = 0$ if and only if $l^{*}(w) = 0$. Suppose that $l^{*}(w) \neq 0$. Then we obtain
$$\{ (P,w)  \: : \: P \in {\mathcal L}_{\mathbb Z}(A) \} \: = \: \{ (P,w) \: : \:  P \in {\mathcal L}^{n}_{\mathbb Z}(A) \}
 \: = \: \langle \, \{ (P, w) \: : \: P \in \mathcal X , \: \mathcal X \: \mbox{generating set for} \,
         {\mathcal L}^{n}_{\mathbb Z}(A) \} \rangle \, .$$
The ideal $\langle \,  n_{1}, n_{2}, \ldots n_{k}  \, \rangle$ generated by given integers
$n_{1}, n_{2}, \ldots , n_{k}$ is equal to $ \langle \gcd \{ n_{1}, n_{2}, \ldots n_{k} \} \rangle$.
Thus if we choose $\mathcal X = \{ l(u) \: : \: \: u \in {A}^{\, n} \}$ then by (\ref{defl*}) and the definition of $c(w)$ we obtain
$c(w) =  \gcd \, \{ (l(u), w) \: : \: u \in {A}^{\, n} \} = \gcd \, \{ ({l^{*}}(w), u) \: : \: u \in {A}^{\, n}\}$, as required.
If, on the other hand, we choose $\mathcal X = \{ P_{l_{1}}, \, P_{l_{2}}, \, \ldots \, , P_{l_{k}} \}$ then again
$c(w) = \gcd \, \{ (P_{l_{1}}, w), \, (P_{l_{2}}, w), \, \ldots , \, (P_{l_{k}}, w) \}$. Now the triangular form of the equations \eqref{Lyndcoeffw}
for $K = \mathbb Z$ immediately implies that ${\xi}_{1}, {\xi}_{2}, \ldots , {\xi}_{k} \in \mathbb Z$ and finally
$\langle \, (P_{l_{1}}, w), \, (P_{l_{2}}, w), \, \ldots , \, (P_{l_{k}}, w) \, \rangle \: = \:
\langle \, {\xi}_{1}, \, {\xi}_{2}, \, \ldots , \, {\xi}_{k} \, \rangle$.
\end{proof}

Suppose that ${l^{*}}(w) = d_1 u_1 \, + \, d_2 u_2 \, + \,  \cdots \, + \, d_s u_s$, where
$d_{1}, d_{2}, \ldots , d_{s} \in {\mathbb Z}^{*}$ and $u_{1}, u_{2}, \ldots , u_{s} \in A^{n}$, \textit{i.e.}, they are words on the alphabet $A$
of length $n$. Then by Theorem 2.4 $c(w) = \gcd \, ( d_1 , d_2 , \ldots , d_s )$ and by an extension of {\em Bezout's identity}
to more than two integers there exist $k_1 , k_2 , \ldots , k_s \in \mathbb Z$ such that
$c(w) = k_{1} d_{1} \, + \, k_{2} d_{2} \, + \,  \cdots \, + \, k_{s}  d_{s}$. Therefore if we set
$Q_w = k_1 \, l(u_{1}) \, + \, k_2 \, l(u_{2}) \, + \, \cdots \, + \, k_s \, l(u_s)$ by \eqref{defl*} we easily
obtain $(Q_w, w) = c(w)$, as required.

\medskip

Let $m$ be a positive integer with $m > 1$. The natural projection $k \mapsto \overline{k} = k \, (\bmod \, m)$ from $\mathbb Z$ onto
${\mathbb Z}_{m}$ induces a surjective map
$\theta \, : \, {\mathbb Z} \langle A \rangle \: \twoheadrightarrow \: {\mathbb Z}_{m} \langle A \rangle$
that sends a polynomial $P = \sum_{u \in {A}^{*}} (P, u) \, u \in {\mathbb Z} \langle A \rangle$ to the polynomial
${\theta}(P) = \sum_{u \in {A}^{*}} \overline{(P, u)} \, u \in {\mathbb Z}_{m} \langle A \rangle$.
Clearly $\ker{\theta} = (m) \, {\mathbb Z} \langle A \rangle$.
The restriction $\psi$ of $\theta$ to the free Lie ring ${\mathcal L}_{\mathbb Z}(A)$ is also surjective onto
${\mathcal L}_{{\mathbb Z}_{m}}(A)$ with $\ker{\psi} = (m) \, {\mathcal L}_{\mathbb Z}(A)$.
If we denote the left normed Lie bracketing over ${\mathbb Z}_{m}$ and its adjoint by $\overline{l}$ and
$\overline{l^{*}}$, respectively, it is easy to see that for each word $u$ in
${A}^{*}$ we have ${\psi}(l(u)) = {\overline{l}}(u)$. From this we can show that
${\theta}({l^{*}}(w)) = {\overline{l^{*}}}(w)$, for each word $w$ in ${A}^{*}$.
Indeed,
${\theta}({l^{*}}(w)) =  {\theta} ( \sum_{u} ( {l^{*}}(w), u) \, u )
                    =   \sum_{u} \overline{( {l^{*}}(w), u)} \, u
                    =   \sum_{u} \overline{( l(u), w)} \, u
                    =   \sum_{u} (\overline{l}(u), w) \, u
                    =   \overline{{l^{*}}}(w)$.

\begin{theorem}\label{suppm|c(w)}
\begin{description}
\item[\em{(i)}]
The complement of the support of the free Lie algebra
${\mathcal L}_{{\mathbb Z}_{m}}(A)$ over the ring ${\mathbb Z}_{m}$ of integers $\bmod \: m$ consists of
all words $w$ such that $m \mid c(w)$ or equivalently those words with the property that the polynomial
$l^{*}(w)$ lies in $(m) \, {\mathbb Z} \langle A \rangle$.

\item[\em{(ii)}]

A pair of words $u, \, v$ is twin (respectively anti-twin) with respect to
${\mathcal L}_{{\mathbb Z}_{m}}(A)$ if and only if the polynomial $l^{*}(u) - l^{*}(v)$
(respectively $l^{*}(u) + l^{*}(v)$) lies in $(m) \, {\mathbb Z} \langle A \rangle$.
\end{description}
\end{theorem}

\begin{proof}
(i)  By Lemma \ref{kerl*supp}\,(i) applied for $K = {\mathbb Z}_{m}$, the complement of the support of ${\mathcal L}_{{\mathbb Z}_{m}}(A)$ is equal to
$\{ w \in A^{*} \: : \: \overline{l^{*}}(w) \, = \, \overline{0} \}$. Since $\ker \theta = (m) \, {\mathbb Z} \langle A \rangle$ and
${\theta}({l^{*}}(w)) = {\overline{l^{*}}}(w)$ the latter
is equal to $\{ \, w \in A^{*} \: : \: {l^{*}}(w) \in (m) \, {\mathbb Z} \langle A \rangle \}$.
If ${l^{*}}(w) = d_1 u_1 \, + \, d_2  u_2 \, + \, \cdots \, + \, d_s u_s$, for
$d_{1}, d_{2}, \ldots , d_{s} \in {\mathbb Z}^{*}$ (that depend on the word $w$) then $m \mid d_{i}$
for all $i \in \{1, 2, \ldots ,s \}$ if and only if $m \mid \gcd ( d_1 , d_2 , \ldots , d_s )$ which, by
Theorem \ref{c(w)gcd}, is equal to $c(w)$. On the other hand, it is clear that
$l^{*}(w) = 0$ is equivalent to $c(w) = 0$. In any case ${l^{*}}(w) \in (m) \, {\mathbb Z} \langle A \rangle$
if and only if $m \mid c(w)$ and the result follows.

(ii) It follows similarly from part (i) and Lemma \ref{kerl*supp}\,(ii).
\end{proof}

\medskip

Let us see now the way the Sch\"utzenberger problems relate to the {\em Pascal triangle} $\bmod \, m$.
When $m = p$, a prime number, an old result due to E. Lucas \cite{Luca} known as the {\em Lucas correspondence theorem}
(for a nice exposition of this see \cite{Fine}) asserts that if $n$ and $r$ have expansions in base $p$ respectively given by
$n = \sum_{q \geq 0} n_{q} p^{\, q}$ and $r = \sum_{q \geq 0} r_{q} p^{\, q}$ with $n_{q}, r_{q} \in \{ 0,1, \ldots p-1 \}$, then
\[ { n \choose r }  \equiv \prod_{q \geq 0} { n_{q} \choose r_{q} } \, (\bmod \, p). \]

By another old result of E. Kummer, known as {\em Kummer's lemma} \cite{Kumm} the highest power of
a prime $p$ dividing $\displaystyle {{ k  + l } \choose k}$ is equal to the number of carries in the $p$-ary
addition of $k$ and $l$.
This enables us to solve directly the Sch\"utzenberger problems for words $w$ of the form $w = a^{k}ba^{l}$.

\begin{lemma}\label{l*1b}
{\em Let $k$ and $l$ be non-negative integers which are not both equal to zero and
$m$ be a positive integer with primary decomposition $m =  p_{1}^{e_{1}} p_{2}^{e_{2}} \cdots p_{s}^{e_{s}}$. Then}
\begin{description}
\item[\em{(i)}]
\begin{equation*}
l^{*}(a^{k}ba^{l}) \: = \: {(-1)}^{k} {{k  + l } \choose k} l^{*}(ba^{k + l})
                   \: = \: {(-1)}^{k} {{k  + l } \choose k} \, \{ ba^{k + l } - aba^{k + l - 1} \} .
\end{equation*}
\item[\em{(ii)}]
The word $a^{k}ba^{l}$ does not lie in the support of the free Lie algebra ${\mathcal L}_{{\mathbb Z}_{m}}(A)$ if and only if
for each $i \in \{1, 2, \ldots , s \}$ the number of carries in the $p_{i}$-ary addition of $k$ and $l$ is at least $e_{i}$.
\end{description}
\end{lemma}

\begin{proof}
(i) First we show that $l^{*}(ba^{m})  =  ba^{m } - aba^{m - 1}$ by an easy induction on $m$. Assuming, without loss of generality,
that $k \geq 1$ the recursive definition of $l^{*}$ yields $l^{*}(a^{k}ba^{l}) = \{ l^{*}(a^{k}ba^{l -1}) - l^{*}(a^{k -1}ba^{l}) \} \, a$.
The proof is completed by induction on $k + l$ and the recursive definition of binomial coefficients.

(ii) By Theorem \ref{suppm|c(w)} a word $w$ lies in the complement of the support of the free Lie algebra
${\mathcal L}_{{\mathbb Z}_{m}}(A)$ if and only if $p_{i}^{e_{i}} \mid c(w)$ for each $i \in \{1, 2, \ldots , s \}$.
Theorem \ref{c(w)gcd} and part (i) yield $c(a^{k}ba^{l}) = {{k  + l } \choose k}$ and the result follows by Kummer's lemma.
\end{proof}

Note that in the binary case the condition of Lemma \ref{l*1b}\,(ii) simply means that if $k$ and $l$ are written in base $2$ as
$k  = \sum_{q \geq 0} {k}_{q} 2^{q}$ and $l = \sum_{q \geq 0} {\lambda}_{q} 2^{q}$ with ${k}_{q}, {l}_{q} \in \{ 0,1 \}$,
there exists at least one position $q$ where ${k}_{q} = {l}_{q} = 1$.

\medskip

Let us now discuss Problem \ref{twinantiZ}. In view of Lemma \ref{kerl*supp}\,(ii), given a pair of words $u, \, v$ one has to check
whether $l^{*}(u) = l^{*}(v)$ (respectively $l^{*}(u) = - l^{*}(v)$) for the pair to be twin (respectively anti-twin).
There are some trivial solutions of this problem, namely when $l^{*}(u) = l^{*}(v) = 0$, \textit{i.e.}, when both $u$ and $v$
are either powers of a letter with exponent larger than one or palindromes of even length. So let us suppose that both words do
lie in the support of the free Lie ring. In view of Lemma \ref{l*reversal} we propose the following conjecture.

\smallskip

\begin{conjecture}\label{conjtwinantiZ}
Let $l^{*}$ be the adjoint endomorphism of the left normed Lie bracketing of the free Lie ring and let
$u$ and $v$ be words of common length $n$ such that both $l^{*}(u)$ and $l^{*}(v)$ are non-zero. Then
\begin{description}
 \item[\em{(i)}]
$l^{*}(u) = l^{*}(v)$ if and only if $u = v$ or $n$ is odd and $u = \tilde{v}$.
 \item[\em{(ii)}]
$l^{*}(u) = - \, l^{*}(v)$ if and only if $n$ is even and $u = \tilde{v}$.
\end{description}
\end{conjecture}

\smallskip

\begin{rtheorem}\label{reducetoalphof2}
It suffices to prove Conjecture \ref{conjtwinantiZ} for an alphabet of two letters.
\end{rtheorem}

The proof of this result will occupy the remaining of this section. Let us first see how $l^{*}$ is affected by alphabetic substitutions.
Consider two finite alphabets $A$ and $\Sigma$ with $|A| \geq |{\Sigma}| \geq 2$ and a mapping
$\phi$ from $A$ onto ${\Sigma}$. This induces a surjective {\em literal} morphism (\textit{i.e.}, a morphism such that ${\phi}(a) \in \Sigma$, for each
$a \in A$), also denoted by $\phi$, from $A^{*}$ onto ${\Sigma}^{*}$ which in turn can be extended linearly to an algebra surjective
homomorphism - still denoted by $\phi$ - from $K {\langle A \rangle}$ onto $K {\langle \Sigma \rangle}$. Let $l^{*}_{A}$ and $l^{*}_{\Sigma}$
denote the adjoint endomorphism of the left normed Lie bracketing of the free Lie algebras ${\mathcal L}_{K}(A)$
and ${\mathcal L}_{K}(\Sigma)$, respectively.

\begin{lemma}\label{l*2alph}
Let $A$ and $\Sigma$ be two finite alphabets with $|A| \geq |{\Sigma}| \geq 2$  and $\phi$ be a fixed literal morphism
from $A^{*}$ onto ${\Sigma}^{*}$.
\begin{description}
\item[\em{(i)}]
The algebra homomorphisms ${\phi} \, l^{*}_{A}$ and $l^{*}_{\Sigma} \, {\phi}$ from $K {\langle A \rangle}$ to
$K {\langle {\Sigma} \rangle}$ are identical and ${\phi}(\ker l^{*}_{A}) \subseteq \ker l^{*}_{\Sigma}$.
\item[\em{(ii)}]
If $|A| = |{\Sigma}|$ then $\ker l^{*}_{A}$ is mapped bijectively onto $\ker l^{*}_{\Sigma}$ under $\phi$.
\item[\em{(iii)}]
${\phi}(\ker l^{*}_{A}) \: = \: \ker l^{*}_{\Sigma}$ but $\ker l^{*}_{A}$ is a proper subset of
${\phi}^{-1}({\ker l^{*}_{\Sigma}})$ if $|A| > |{\Sigma}|$.
\end{description}
\end{lemma}

\begin{proof}
(i) It suffices to show that $l^{*}_{\Sigma}({\phi}(w)) = {\phi}(l^{*}_{A}(w))$ for each word
$w \in A^{*}$. This follows by an easy induction on $|w|$ using the recursive definition \eqref{defrecl*} of $l^{*}$.
The inclusion ${\phi}(\ker l^{*}_{A}) \subseteq \ker l^{*}_{\Sigma}$ follows easily.

(ii) We apply the result of part (i) for the literal morphisms $\phi : A \twoheadrightarrow \Sigma$ and
${\phi}^{-1} : \Sigma \twoheadrightarrow A$ and obtain ${\phi}(\ker l^{*}_{A}) \subseteq \ker l^{*}_{\Sigma}$ and
${\phi}^{-1}(\ker l^{*}_{\Sigma}) \subseteq \ker l^{*}_{A}$ which clearly implies the required result.

(iii) It remains to show that ${\phi}(\ker l^{*}_{A}) \supseteq \ker l^{*}_{\Sigma}$.
Let $Q = \sum_{i} {\lambda}_{i} v_{i}$ be a polynomial in $\ker l^{*}_{\Sigma}$ and
${\Sigma}_{Q} = \bigcup_{i} \, alph(v_{i})$. For each letter $b \in {\Sigma}_{Q}$ choose a unique letter $a \in A$ such that
${\phi}(a) = b$ and let $A_{Q}$ be the subset of $A$ consisting of all those chosen letters.
Clearly the restriction $\widehat{\phi}$ of $\phi$ to $A_{Q}$ is a bijection from $A_{Q}$ onto ${\Sigma}_{Q}$.
Consider the extension of $\widehat{\phi}$ - denoted by the same symbol - to ${A_{Q}}^{*}$ which is a bijective
homomorphism to ${{\Sigma}_{Q}}^{*}$.
Let $P = \sum_{i} {\lambda}_{i} {\widehat{\phi}}^{-1}(v_{i})$.
Then clearly $\widehat{\phi}(P) = Q$, so that ${\phi}(P) = Q$. Furthermore, since $Q \in \ker l^{*}_{{\Sigma}_{Q}}$ and
$|A_{Q}| = |{\Sigma}_{Q}|$ part (ii) applied for the bijection $\widehat{\phi}$ implies that
$P$ also lies in $\ker l^{*}_{A_{Q}}$ and hence in $\ker l^{*}_{A}$.

Suppose now that $|A| > |{\Sigma}|$. We can consider three distinct letters $a, b, c \in A$
and two distinct letters $e, f \in \Sigma$ in such a way that, without loss of generality,
${\phi}(a) = e$ and ${\phi}(b) = {\phi}(c) = f$.
Consider the word $w = abca$. Then $w \in {\phi}^{-1}({\ker l^{*}_{\Sigma}})$ since clearly ${\phi}(w) = ef^{2}e$ is
a palindrome of length $4$, but $w \not\in \ker l^{*}_{A}$ since $l_{A}^{*}(w) = abca - baca - 2bca^{2} + 2cba^{2} + caba - acba \neq 0$.
\end{proof}

\smallskip

Consider a two-lettered alphabet $\Sigma = \{0,1\}$. For each subset $B$ of $A$ let ${\phi}_{B}$ be the literal morphism from $A^{*}$ onto ${\Sigma}^{*}$ defined as ${\phi}_{B}(a)  = 1$, when $a \in B$ and ${\phi}_{B}(a)  = 0$, otherwise.
For brevity when $B = \{b\}$ we write ${\phi}_{b}$ instead of ${\phi}_{\{ b \}}$. Let us also denote the set of palindromes of length $n$
in $A^{*}$ and ${\Sigma}^{*}$ by $Pal_{n}(A)$ and $Pal_{n}({\Sigma})$, respectively.

\begin{lemma}\label{alphsubst}
Let $u, v \in A^{*}$ of common length and $a, b$ be distinct elements of $A$.
\begin{description}
\item[\em{(i)}]
${\phi}_{a}(u) = {\phi}_{a}(v)$, for each $a \in A$, if and only if $u = v$.
\item[\em{(ii)}]
If ${\phi}_{a}(u) \in Pal_{n}({\Sigma})$ for each $a \in A$, then $u \in Pal_{n}(A)$.
\item[\em{(iii)}]
If ${\phi}_{\{a,b\}}(u) = {\phi}_{\{a,b\}}(v)$ and ${\phi}_{a}(u) = {\phi}_{a}(v)$ then ${\phi}_{b}(u) = {\phi}_{b}(v)$.
\item[\em{(iv)}]
If ${\phi}_{\{a,b\}}(u)$ and ${\phi}_{a}(u)$ lie in $Pal_{n}({\Sigma})$ then ${\phi}_{b}(u)$ does also.
\end{description}
\end{lemma}

\begin{proof}
For (i) and (iii) it suffices to check the case where $u, v \in A$. Then (i) follows directly from the definition of the morphism ${\phi}_{a}$.
For (iii) if $u, v \not\in \{a,b\}$ the result is clear, whereas if $u, v \in \{a,b\}$ we necessarily obtain $u = v = a$ or $u = v = b$ and
the result follows.
Parts (ii) and (iv) follow respectively from (i) and (iii) for $v = \tilde{u}$ and the fact that $\widetilde{{\phi}(w)} = {\phi}(\tilde{w})$ for
each literal morphism $\phi$ and each word $w \in A^{*}$ which is easily proved by induction on $|w|$.
\end{proof}

\medskip

{\em Proof of Reduction Theorem \ref{reducetoalphof2}.}\,
Suppose that Conjecture \ref{conjtwinantiZ} is true for a two lettered alphabet $\Sigma = \{0,1\}$. We will show that it also holds for any finite alphabet $A$ with $|A| > 2$. Suppose that $u, v \in A^{*}$ of common length $n > 1$ and that both polynomials $l^{*}(u)$ and $l^{*}(v)$ are non-zero.

\smallskip

{\bf (1) $\mathbf {l^{*}(u) = l^{*}(v) }$}. Then $u - v \in \ker l^{*}_{A}$ and for any subset $B$ of $A$
${\phi}_{B}(u - v) \in \ker l^{*}_{\Sigma}$ by Lemma \ref{l*2alph}\,(i), so we obtain $l^{*}({\phi}_{B}(u)) = l^{*}({\phi}_{B}(v))$, for each
$B \subseteq A$. We consider two cases.

\smallskip

{\em (i) $n$ is even}.
Our aim is to show that $u = v$.
If $l^{*}({\phi}_{B}(u)) = l^{*}({\phi}_{B}(v)) \neq 0$, then Conjecture \ref{conjtwinantiZ} yields
${{\phi}_{B}}(u) = {{\phi}_{B}}(v) \notin Pal_{n}({\Sigma})$. If, on the other hand, $l^{*}({\phi}_{B}(u)) = l^{*}({\phi}_{B}(v)) = 0$
both ${\phi}_{B}(u)$ and ${\phi}_{B}(v)$ lie in $Pal_{n}({\Sigma})$. Restricting ourselves initially to singleton subsets $B = \{ a \}$
we define, for our fixed words $u$ and $v$, the sub-alphabets $C$ and $D$ of $A$ as
\[ C  =  \{ a \in A \: : \: {\phi}_{a}(u), {\phi}_{a}(v) \in  Pal_{n}({\Sigma}) \} \quad \mbox{and} \quad
   D  =  \{ a \in A \: : \: {\phi}_{a}(u) = {\phi}_{a}(v)  \not \in Pal_{n}({\Sigma}) \}. \]
By construction $A = C \cup D$ and $C \cap D = \emptyset$.
Suppose that $C = \emptyset$. Then $A = D$ and the result follows immediately by Lemma \ref{alphsubst}\,(i).
If $C \neq \emptyset$ then also $D \neq \emptyset$, since otherwise $A = C$ and by Lemma \ref{alphsubst}\,(ii) it would follow that $u, \, v \in Pal_{n}(A)$, a fact which clearly contradicts our original assumption that both $l^{*}(u)$ and $l^{*}(v)$ are non-zero.
Let $c$ be an arbitrary element of $C$.
In view of Lemma \ref{alphsubst}\,(i) it suffices to show that ${\phi}_{c}(u) = {\phi}_{c}(v)$. Let $d \in D$ and set $B = \{ c, d \}$. Then either ${\phi}_{\{c,d\}}(u)$ and ${\phi}_{\{c,d\}}(v)$ lie in $Pal_{n}({\Sigma})$ or ${\phi}_{\{c,d\}}(u) = {\phi}_{\{c,d\}}(v) \notin Pal_{n}({\Sigma})$. In the former case the fact that ${\phi}_{c}(u), {\phi}_{c}(v) \in Pal_{n}({\Sigma})$ and Lemma \ref{alphsubst}\,(iv) yield
${\phi}_{d}(u), {\phi}_{d}(v) \in Pal_{n}({\Sigma})$, which contradicts the fact that $d \in D$.
In the latter one Lemma \ref{alphsubst}\,(iii) finally yields ${\phi}_{c}(u) = {\phi}_{c}(v)$, as required.

\smallskip

{\em (ii) $n$ is odd}. Our aim is to show that either $u=v$ or $u = {\tilde{v}}$.
First we show that either ${{\phi}_{B}}(u) = {{\phi}_{B}}(v)$ or ${{\phi}_{B}}(u) = {\phi}_{B}({\tilde v})$.
If $l^{*}({\phi}_{B}(u)) = l^{*}({\phi}_{B}(v)) \neq 0$ this follows immediately by
Conjecture \ref{conjtwinantiZ} which is assumed to hold for $\Sigma$. If, on the other hand, $l^{*}({\phi}_{B}(u)) = l^{*}({\phi}_{B}(v)) = 0$ then both ${{\phi}_{B}}(u)$ and ${{\phi}_{B}}(v)$ lie in $\{ 0^{n}, 1^{n} \}$. We claim that ${{\phi}_{B}}(u) = {{\phi}_{B}}(v)$.
If this is not the case, without loss of generality, ${{\phi}_{B}}(u) = 1^{n}$ and ${{\phi}_{B}}(v) = 0^{n}$. By the definition of ${\phi}_{B}$ it follows that $alph(u) \subseteq B$ and $alph(v) \subseteq A \setminus B$. On the other hand having assumed that $l^{*}(u) = l^{*}(v) \neq 0$
we also get $alph(u) = alph(v)$ and we reach a contradiction.
For our fixed pair $(u,v)$ define the sub-alphabets $E$ and $F$ of $A$ as
\[ E  =  \{ a \in A \: : \: {\phi}_{a}(u) = {\phi}_{a}(v)  \} \quad  \mbox{and} \quad
   F  =  \{ a \in A \: : \: {\phi}_{a}(u) = {\phi}_{a}({\tilde{v}}) \} . \]
It follows that $A = E \cup F$. It remains to show that either $E = A$ or $F = A$.
Suppose, for the sake of contradiction, that $e \in E \setminus F$ and $f \in F \setminus E$.
Then if we let $B = \{ e,f \}$ we either get ${\phi}_{\{ e,f \}}(u) = {\phi}_{\{ e,f \}}(v)$ or
${\phi}_{\{ e,f \}}(u) = {\phi}_{\{ e,f \}}({\tilde v})$.
By Lemma \ref{alphsubst}\,(iii) the former equality together with the fact that $e \in E$ yields ${\phi}_{f}(u) = {\phi}_{f}(v)$,
so that $f$ will also lie in $E$, which is a contradiction since we took $f \in F \setminus E$. We get a
similar contradiction starting from the latter equality.

\medskip

{\bf (2) $\mathbf {l^{*}(u) = - \, l^{*}(v) }$}. It means that $u + v \in \ker l^{*}_{A}$ so by Lemma \ref{l*2alph}\,(i)
${\phi}_{B}(u + v) \in \ker l^{*}_{\Sigma}$ and therefore $l^{*}({\phi}_{B}(u)) = - \, l^{*}({\phi}_{B}(v))$, for each subset $B$ of $A$.
The case where $n$ is even is dealt in an analogous way as in case 1(i) before; we leave the details to the reader.

If $n$ is odd, by Conjecture \ref{conjtwinantiZ}, we can not have $l^{*}({\phi}_{B}(u)) = - \, l^{*}({\phi}_{B}(v)) \neq 0$.
Thus for each $B \subseteq A$ both ${\phi}_{B}(u)$ and  ${\phi}_{B}(v)$ lie in $\{ 0^{n}, 1^{n} \}$.
This will hold in particular, for each singleton $B = \{a\}$, with $a \in A$.
Then if ${\phi}_{a}(u) = 1^{n}$ for some $a \in A$ we obtain $u = a^{n}$, a contradiction.
Therefore ${\phi}_{a}(u) = {\phi}_{a}(v) = 0^{n}$ for each $a \in A$. Then $alph(u) \cap A = \emptyset$, which also clearly can not hold.
\hfill $\square$

\bigskip


\section{Calculation of $\mathbf {l^{*}}$}\label{l*shuffle}

We will now generalize the effective definition (\ref{defrecl*}) of $l^*$ using the shuffle product of words and calculate the
polynomial $l^{*}(w)$ recursively in terms of all factors $u$ of fixed length $r \geq 1$ of $w$.
By a factor of a word $w$ we mean a word $u$ such that there exist $s, t \in A^{*}$ with $w = sut$.

\begin{proposition}\label{recl*shuffle}
Let $w$ be a word and $r$ be a positive integer with $r \leq |w|$.
Consider the set of all factors $u$ of length $r$ of $w$. Then
\[l^{*}(w) \: = \: \sum_{{w = sut} \atop {|u|=r}} \, l^{*}(u) \, (-1)^{|s|} \, \{ \tilde {s} \, \sh \, \,  t \} \, . \]
\end{proposition}

\begin{proof}
Let $|w| = n$. We argue by induction on $k = n - r$. Clearly $0 \leq k \leq n - 1$.
Note that for $k = 0$ the result is trivial since $s = t = \epsilon$.
Also for $k = 1$ it follows from the recursive definition (\ref{defrecl*}) of $l^*$ since the factors of
length $n - 1$ of $w = aub$ are just the words $au$ and $ub$.

Let $\{ u_{i} \, : \, 0 \leq i \leq n-r \}$ be the set of all $n - r + 1$ consecutive factors of length $r \geq 2$ of $w$.
We let $u_{i} = a_{i} v_{i} b_{i}$ with $a_{i}, \, b_{i} \in A$, $v_{i} \in {A}^{*}$.
Note that $v_{i} b_{i} = a_{i+1} v_{i+1}, \: 0 \leq i \leq n-r-1$.

Suppose that the result holds for the factors $u_i$. We have to show that it also holds for factors
of length $r - 1$ of $w$.
We let $w = s_{i} u_{i} t_{i}$ for $0 \leq i \leq n-r$, where $s_{0} = t_{n-r} = \epsilon$,
$s_{1} = a_{0}$ and $t_{n-r-1} = b_{n-r}$. Then $| \widetilde {s_{i}} | = |s_{i}| = i$, for each
$0 \leq i \leq n-r$ and our induction hypothesis for factors of length $r$ yields
\begin{equation*}
l^{*}(w)  =  l^{*}(u_{0}) t_{0} \, + \,
               \sum_{i=1}^{n-r-1} \, l^{*}(u_{i}) \, (-1)^{i} \, \{ \widetilde {s_{i}} \, \sh \, \,  t_{i} \}
               \, + \, l^{*}(u_{n-r}) \, (-1)^{n-r} \, \widetilde{s_{n-r}} \, .
\end{equation*}
We apply the recursive formula (\ref{defrecl*}) on factors $u_{i} = a_{i} v_{i} b_{i}$ and obtain
\begin{eqnarray*}
l^{*}(w) & = & \Big[ l^{*}(a_{0} v_{0}) b_{0}  \, - \, l^{*}(v_{0} b_{0}) a_{0} \Big] \, t_{0}  \, + \,
               \sum_{i=1}^{n-r-1} \, \Big[ l^{*}(a_{i} v_{i}) b_{i} \, - \, l^{*}(v_{i} b_{i}) a_{i} \Big]
                \, (-1)^{i} \, \{ \widetilde {s_{i}} \, \sh \, \,  t_{i} \}  \\
         &   & \, + \, \Big[ l^{*}(a_{n-r} v_{n-r}) b_{n-r}  \, - \, l^{*}(v_{n-r} b_{n-r}) a_{n-r} \Big] \,
                (-1)^{n-r} \, \widetilde{s_{n-r}} \, .
\end{eqnarray*}
Since $v_{i} b_{i} = a_{i+1} v_{i+1}$ for $0 \leq i \leq n-r-1$, grouping all elements of the form
$a_{i} v_{i}$ for $0 \leq i \leq n-r$ we obtain
\begin{eqnarray*}
\displaystyle l^{*}(w) & = & l^{*}(a_{0} v_{0}) \, (-1)^{0} b_{0} t_{0}   \, +  \,
               l^{*}(a_{1} v_{1}) \, (-1)^{1} \,  \{ a_{0} t_{0} \, + \,
               b_{1} \, ( a_{0} \sh \, t_{1} ) \} \, + \, \\
         &   & \sum_{i=2}^{n-r-1} \, l^{*}(a_{i} v_{i}) \, (-1)^{i} \,
    \{ b_{i} \, ( \widetilde{s}_{i} \sh \, t_{i} ) \, + \, a_{i-1} \, ( \widetilde{s_{i-1}} \sh \, t_{i-1} ) \} \, + \, \\
         &   &  l^{*}(a_{n-r} v_{n-r}) \, (-1)^{n-r} \, \{ b_{n-r} \widetilde{s_{n-r}} \, + \,
                                          a_{n-r-1} \, ( \widetilde{s_{n-r-1}} \sh \, t_{n-r-1} ) \} \, + \, \\
         &   & l^{*}(v_{n-r} b_{n-r}) \, (-1)^{n-r+1} \, a_{n-r} \widetilde{s_{n-r}} \, .
\end{eqnarray*}
Since $s_{i} = s_{i-1} a_{i-1}$ and $t_{i-1} = b_{i} t_{i}$ for $1 \leq i \leq n-r$ (note the extreme cases
$s_{1} = a_{0}$ for $i = 1$ and $t_{n-r-1} = b_{n-r}$ for $i = n-r$), from the recursive definition (\ref{recdefshuffle}) of the shuffle product
we get
$$b_{i} \, ( \widetilde{{s}_{i}} \sh \, t_{i} ) \, + \, a_{i-1} \, ( \widetilde{s_{i-1}} \sh \, t_{i-1} )
\, = \, a_{i-1} \widetilde{s_{i-1}} \sh \, b_{i} t_{i} \, = \, \widetilde{{s}_{i}} \sh \, b_{i} t_{i},$$
for all $i \in \{ 1, \ldots , n-r \}$.
Then we immediately obtain
\begin{eqnarray*}
l^{*}(w) & = & l^{*}(a_{0} v_{0}) \, (-1)^{0} \{ {\epsilon} \sh \, b_{0} t_{0} \}  \, +  \,
               \sum_{i=1}^{n-r} \, l^{*}(a_{i} v_{i}) \, (-1)^{i} \, \{ \widetilde{{s}_{i}} \sh \, b_{i} t_{i} \} \, + \, \\
         &   & l^{*}(v_{n-r} b_{n-r}) \, (-1)^{n-r+1} \, \{ \widetilde{s_{n-r} a_{n-r}} \sh \, {\epsilon} \} \, ,
\end{eqnarray*}
which is precisely the required summation for factors of length $r-1$ we had aimed for.
\end{proof}

{\em Remarks}. The case where the factors $u$ of the word $w$ are letters (\textit{i.e.}, we are at the bottom
level $r=1$) seems to be known in the literature (see \cite[Ex. 4.6.5, p.126]{Diek + Rozen})
even for the broader class of free partially commutative Lie algebras.
On the other hand, Proposition \ref{recl*shuffle} clearly does not hold if we consider trivial factors of $w$, \textit{i.e.}, factors of length $r=0$.
In this case the identity
$\displaystyle 0 \: = \: \sum_{w = st} \,  (-1)^{|s|} \, \{ \tilde {s} \, \sh \, \,  t \}$,
due to W. Schmidt, holds; see \cite[\S 1.6.4]{Reut}.

\bigskip

We will now use Proposition \ref{recl*shuffle} to reobtain - in a non ad hoc way - the result by Duchamp and Thibon
\cite[\S 3, p.124]{Duch + Thib} for the calculation of the support of the free Lie ring.

\smallskip

\begin{theorem}\label{DuchThib}
The words that vanish under the adjoint endomorphism $l^{*}$ of the left normed Lie bracketing $\, l$
of the free Lie ring are either powers of a single letter with exponent greater than one or
palindromes of even length.
\end{theorem}

\begin{proof}
If $w = a^{n}$, where $n \geq 2$, then $l^{*}(w) = 0$, since
${l^{*}}(w) \, = \, {l^{*}}(a^{n-1}) \, a \, - \, {l^{*}}(a^{n-1}) \, a$.
For a palindrome $w$ of even length Lemma 2.1 yields $2 \, {l^{*}}(w) = 0$, hence ${l^{*}}(w) = 0$
since we are in characteristic zero.

\medskip

For the other direction of the theorem consider a word $w$ such that $l^{*}(w) = 0$. We will argue by induction
on the length $|w| = n$. For $n = 2$ we necessarily get $w = a^{2}$ because if $w = ab$ with $a \neq b$ then
$l^{*}(ab) = ab - ba \neq 0$, so the result follows trivially.
Let $n > 2$. We consider two cases.

\smallskip

{\bf (1)} $\mathbf{w = aub, \: (a, b \in A, \: a \neq b)}$.
Then $0 \, = \, l^{*}(w) = l^{*}(au) \, b \, - \, l^{*}(ub) \, a$. Since $a \neq b$ we get
$l^{*}(au) = l^{*}(ub) = 0$. By our induction hypothesis we have to consider two subcases.

(i) \,If at least one of the words $au$ and $ub$ (without loss of generality say $au$) is a power of a single
letter we obtain $au = a^{n-1}$, so that $ub = a^{n - 2} \, b$ which is neither a power of a single letter nor a palindrome of even length,
so we reach a contradiction.

(ii) \,If both $au$ and $ub$ are palindromes of even length and not powers of a single letter we must have
$au = s \tilde{s}$, for some $s \in A^{+}$, so that there exists a $t \in A^{*}$ with $s = at$ and $u = t \tilde{t} a$.
But then the word $ub = t \tilde{t} a b$ can not be a palindrome of even length since the number of occurrences $|t \tilde{t} a|_{a}$
of the letter $a$ in the word $t \tilde{t} a$ is equal to $2 |t|_{a} + 1$, an odd positive integer and we obtain another contradiction.

\medskip

{\bf (2)} $\mathbf{w = a^{k} \, bvc \, a^{l}, \: (a,b,c \in A, \: b, c \neq a)}$.
We consider all factors $x$ of length $|v| + 2$ of $w$. Then Proposition \ref{recl*shuffle} yields
\begin{eqnarray*}
l^{*}(w) & = & l^{*}(bvc) \, (-1)^{k} \, \{ a^{k} \sh \, a^{l} \} \quad + \,
               \sum_{{w = sxt} \atop {(s,t) \neq (a^{k},a^{l})}} \,
               l^{*}(x) \, (-1)^{|s|} \, \{ \tilde{s} \sh \, t \} \\
         &  = & (-1)^{k} {{k + l} \choose k} \, l^{*}(bvc) \, a^{k+l} \quad + \,
               \sum_{{w = sxt} \atop {(s,t) \neq (a^{k},a^{l})}} \,
               l^{*}(x) \, (-1)^{|s|} \, \{ \tilde{s} \sh \, t \}.
\end{eqnarray*}
For any factorization $w = sxt$ other than the one where $(s,x,t) = (a^{k}, bvc, a^{l})$, each shuffle
of $\tilde{s}$ and $t$ contains other letters except $a$ so it is different from $a^{k+l}$, which appears only as a shuffle of $a^k$ and $a^l$.
From this we deduce that the monomials of $l^{*}(w)$ such that a power of the letter $a$ appears as a
right factor with maximum possible exponent are precisely the monomials in
${{k + l} \choose k} \, l^{*}(bvc) \, a^{k+l}$.
Now, the assumption $l^{*}(w) = 0$ yields ${{k+l} \choose l} \, l^{*}(bvc) = 0$ and since we are in characteristic
zero we immediately obtain $l^{*}(bvc) = 0$. Then by our induction hypothesis we get $b = c$ and $v$ is a power of
$b$ or a palindrome of even length. Two subcases have to be considered.

(i) \,If $k \neq l$ (without loss of generality say $k < l$) we consider factors $y$ of length $|v| + l + 1$
of $w$ and apply again Proposition \ref{recl*shuffle}. We obtain
\begin{equation}
 l^{*}(w) \: = \: \sum_{{w = syt} \atop {(s,t) \neq (a^{k}b, {\epsilon})}}
 l^{*}(y) \, (-1)^{|s|} \, \{ \tilde{s} \sh \, t \} \quad + \quad l^{*}(vba^{l}) \, (-1)^{k+1} \, b \, a^{k}. \label{vbal}
\end{equation}
In this case every shuffle of the words $\tilde{s}$ and $t$ from each term of the first summand of \eqref{vbal} will be equal to $a^{k+1}$.
Assuming that $l^{*}(w) = 0$ yields $(-1)^{k+1} l^{*}(vba^{l}) ba^{k} \, + \, P a^{k+1} \, = \, 0$,
for some polynomial $P \in {\mathbb Z}{\langle A \rangle}$. But then $l^{*}(vba^{l}) = P = 0$, so by our induction hypothesis the word $vba^{l}$
has to be a palindrome of even length. When $v$ is a power of $b$ this clearly can not happen. It remains to check the case where both $bvb$ and $vba^{l}$ are simultaneously palindromes of even length. Then $|v|_{b}$ would be an even positive integer in the former case, whereas $|v|_{b} = |vba^{l}|_{b} - 1$ would be odd in the latter; a clear contradiction.

(ii) \,Suppose that $k = l$. If $bvb = b^{q}$ with $q$ an odd positive integer then by assumption
$0 = l^{*}(w) = l^{*}(a^{k} \, b^{q} a^{k}) = 2 \, l^{*}(a^{k} \, b^{q} \, a^{k-1}) \, a$, so that
$l^{*}(a^{k}b^{q}a^{k-1}) = 0$ which is a contradiction since the word $a^{k}b^{q}a^{k-1}$ can not be a palindrome of even length.
So we are finally left with the case where $bvb$ is a palindrome of even length which is what we had originally aimed for.
\end{proof}

\smallskip

\begin{theorem}\label{alphbound}
Let $K$ be a commutative ring with unity and suppose that $l^{*}(w) = 0$ for a word $w \in A^{*}$.
Then $|alph(w)| \leq  {\lceil |w|/2 \rceil}$.
\end{theorem}

\begin{proof}
Consider a word $w$ such that $l^{*}(w) = 0$. We will argue by induction
on the length $|w| = n$. For $n = 2$ we clearly get $w = a^{2}$ and the result follows trivially.
Let $n > 2$. We consider two cases.

\smallskip

{\bf (1)} $\mathbf{w = aub, \: (a, b \in A, \: a \neq b)}$.
Then $0 \, = \, l^{*}(w) = l^{*}(au) \, b \, - \, l^{*}(ub) \, a$. Since $a \neq b$ we get
$l^{*}(au) = l^{*}(ub) = 0$. Our first claim is that both letters $a$ and $b$ have to lie in $alph(u)$.
Indeed, suppose for the sake of contradiction, that - without loss of generality - $a \notin alph(u)$.
Let us consider all factors of length $1$ of $w = aub$ and apply Proposition \ref{recl*shuffle}. We obtain
\begin{equation*}
 \displaystyle l^{*}(w)  = \, a \, ub \quad + \, \sum_{{w = sct} \atop {c \, \in \, alph(ub)}} \,
              (-1)^{|s|} \, c \, \{ \tilde{s} \sh \, t \} \, .
\end{equation*}
Since $a \notin alph(u)$ and $a \neq b$, we have $c \neq a$, hence the only monomial of $l^{*}(w)$ that starts with the letter $a$
is the word $aub$ which cannot be cancelled and therefore $l^{*}(w) \neq 0$, a contradiction.

Having obtained that $a, b \in alph(u)$, we get $alph(w) = alph(u)$.
Our result then follows since, by our induction hypothesis, $|alph(u)| \leq \lceil (n-2)/2 \rceil \leq \lceil n/2 \rceil$.

\medskip

{\bf (2)} $\mathbf{w = aua, \: (a \in A)}$.
We define $r(w) \: = \: \max \, \{ \, |s| \: : \: w = sut \, , \: |s| = |t| \, , \: alph(s) = alph(t) \, \}$.
Since $w$ starts and ends with the same letter, $r(w)$ is a well defined positive integer.
Let $p$ and $q$ be respectively the left and the right factor of $w$ of length equal to $r(w)$.
There are three cases to consider: either $w = pq$; $w = pbq$ with $b \in A$; or finally $w = pbucq$, where $u \in A^{*}$,
$b, c$ are distinct letters and at least one of them, without loss of generality say $b$, does not lie in $alph(p)$.
In the first case our result follows immediately since clearly $|alph(w)| = |alph(p)| \leq |p|  = |w|/2 = \lceil |w|/2 \rceil$.
Similarly in the second one $|alph(w)| \leq |alph(p)| + 1 \leq |p| + 1 = \lceil |w|/2 \rceil$.
Finally in the third one $alph(w) = alph(pbuc) = alph(p) \cup alph(buc)$, hence
$|alph(w)| \leq |alph(p)| + |alph(buc)|$. Since $b \neq c$ case (1) yields $|alph(buc)| \leq \lceil |buc|/2 \rceil$, so that
$|alph(w)| \leq |p| + \lceil |buc|/2 \rceil = \lceil (2|p| + |buc|)/2 \rceil = \lceil |w|/2 \rceil$.
\end{proof}

\medskip

\begin{corollary}\label{c(w)=1}
If $w$ is a word of $A^{*}$ with $|alph(w)| > {\lceil |w|/2 \rceil}$ then $c(w) = 1$.
\end{corollary}

\begin{proof}
Suppose that $c(w) = m$ where $m$ is a non-negative integer with $m \neq 1$. If $m=0$ then $l^{*}(w) = 0$ over $\mathbb Z$, hence by
Theorem \ref{alphbound} for $K = \mathbb Z$ we get $|alph(w)| \leq {\lceil |w|/2 \rceil}$, a contradiction.
If $m > 1$ then $l^{*}(w) \in (m) \, {\mathbb Z} \langle A \rangle$ by Theorem \ref{suppm|c(w)}\,(i). Thus if $\overline{l^{*}}$ is the adjoint of
the left normed Lie bracketing $\overline{l}$ over ${\mathbb Z}_{m}$, we get $\overline{l^{*}}(w) = \overline{0}$.
Once more Theorem \ref{recl*shuffle} for $K = {\mathbb Z}_{m}$ yields $|alph(w)| \leq {\lceil |w|/2 \rceil}$, the same contradiction.
\end{proof}

\smallskip

Recall that by Lemma \ref{l*2alph}\,(iii) if $A$ and $\Sigma$ are finite alphabets with $|A| > |{\Sigma}| \geq 2$ and $\phi$ is a literal morphism from $A^{*}$ onto ${\Sigma}^{*}$ then ${\phi}(\ker l_{A}^{*}) = \ker l_{\Sigma}^{*}$ and $\ker l_{A}^{*}$ is a proper subset of
${\phi}^{-1}(\ker l_{\Sigma}^{*})$.
It is worth asking the following: is it possible to have a solution $w$ of the equation $l_{\Sigma}^{*}(w) = 0$ with $alph(w) = \Sigma$
that can not be the image under {\em any} literal surjective morphism $\phi: A^{*} \twoheadrightarrow {\Sigma}^{*}$ of a corresponding
solution $u$ of the equation $l_{A}^{*}(u) = 0$ with $alph(u) = A$?
If $|w| = n$ Theorem \ref{alphbound} implies that our question makes sense when in fact ${\lceil n/2 \rceil} \geq |A| > |{\Sigma}|$.

The following example demonstrates that this is indeed possible.

\begin{example}\label{example}
Set $K = {\mathbb Z}_{2}$, $A = \{a,b,c,d\}$, $\Sigma = \{e,f,g\}$ and $w = efegfef$. Then $w \in \ker l^{*}_{\Sigma}$ but
for each literal morphism $\phi$ from $A^{*}$ onto ${\Sigma}^{*}$ no word $u$ with $alph(u) = A$ and ${\phi}(u) = w$ lies in $\ker l^{*}_{A}$.
\end{example}

\begin{proof}
All calculations are made over ${\mathbb Z}_{2}$. First we show that $l^{*}_{\Sigma}(w) = 0$. Consider all factors of length $3$ in $w$
and apply Proposition \ref{recl*shuffle}.
Since $l^{*}(efe) = l^{*}(fef) = 0$ and $e \sh fef =  fe \sh ef  = efe \sh f = efef + fefe, \,$ we get
$l^{*}_{\Sigma}(w) \: = \: l^{*}_{\Sigma}(feg) \{ e \sh fef \} + l^{*}_{\Sigma}(egf) \{ fe \sh ef \} +
               l^{*}_{\Sigma}(gfe) \{ efe \sh f \}
         \: = \: \{ l^{*}_{\Sigma}(feg) + l^{*}_{\Sigma}(egf) + l^{*}_{\Sigma}(gfe) \} \{ efef + fefe \}$.
The terms in $l^{*}_{\Sigma}(feg) + l^{*}_{\Sigma}(egf) + l^{*}_{\Sigma}(gfe)$ cancel out and therefore
$l^{*}_{\Sigma}(w) = 0$.
Now consider an arbitrary surjective map $\phi$ from $A$ onto $\Sigma$ and an arbitrary word $u$ in ${\phi}^{-1}(\{w\})$ with $alph(u) = A$.
Without loss of generality we may assume that ${\phi}(a) = e, \, {\phi}(b) = f, \, {\phi}(d) = g$ and ${\phi}(c) \in \{e,f\}$. Indeed if
${\phi}(c) = g$ then $alph(u)$ is either $\{a,b,c\}$ or $\{a,b,d\}$ which in both cases is a proper subset of $A$.
We may also assume that ${\phi}(c) = f$; the case ${\phi}(c) = e$ is handled in a similar manner.
Then $u = apadqar$, where $p,q$ and $r$ are letters that lie in $\{b,c\}$. We will show that $l^{*}_{A}(u) \neq 0$.
Suppose the contrary. Since $r \neq a$ we get $l^{*}_{A}(padqar) = 0$, so by Theorem \ref{alphbound} $|alph(padqar)| \leq 3$. On the other hand,
clearly $alph(padqar) = alph(u) = A$, so that $|alph(padqar)| = 4$ and we reach a contradiction.
\end{proof}

\bigskip


\section{Combinatorial interpretation of ${\mathbf l^{*}}$}\label{lnsymm}

It is customary for many problems on free Lie algebras to boil down to
particular combinatorial questions on the group algebra of the symmetric group.
This will also be the case for the Sch\"utzenberger problems.

We start from the {\em place permutation action} of the symmetric group ${\mathfrak S}_{n}$ on $n$
letters, on the set of words of length $n$, where if $w = x_1 \ldots x_n$ and
$\sigma \in {\mathfrak S}_{n}$ we have
$(x_1 \ldots x_n ) \cdot {\sigma} = x_{{\sigma}(1)} \ldots x_{{\sigma}(n)}$.
This is a right action of ${\mathfrak S}_{n}$ that extends by linearity
to a right action of the group ring $K{{\mathfrak S}_{n}}$ on the $n$-th homogeneous component of the free associative
algebra ${K{\langle A \rangle}}$ (e.g., see \cite[\S 8.1, \S 3.3]{Reut}).
Viewing each permutation in ${\mathfrak S}_{n}$ as a word $x_1 x_{2} \ldots x_n$ in $n$ distinct letters, the left normed multi-linear
Lie bracketing of the free Lie algebra, denoted by $l_n$, can be viewed as the element of $K {\mathfrak S}_{n}$ defined by
\begin{equation}
(x_{1} x_{2} \ldots x_{n}) \cdot l_{n} \: = \: l(x_{1} x_{2} \ldots x_{n}) \, . \label{defln}
\end{equation}
For a non-negative integer $k$ let $[k]$ denote the set $\{1, 2, \ldots , k \}$, when $k \geq 1$, or the empty set, when $k = 0$.
A {\em descent} of a permutation $\sigma \in {\mathfrak S}_{n}$ is a position $i \in [n-1]$ for which ${\sigma}(i) > {\sigma}(i+1)$.
Let $D({\sigma})$ be the set of descents of $\sigma$ and for $X \subseteq  [n-1]$ let
$\displaystyle D_{X} = \sum_{D({\sigma}) = X} \!\!\! {\sigma} \, \in K{{\mathfrak S}_{n}}$.
Then the following formulae for $l_n$ are well known (see \cite[Theorem 8.16]{Reut})
\begin{eqnarray}
l_{n} & = & ({\mathbf 1} - {\zeta}_{2} )({\mathbf 1} - {\zeta}_{3}) \cdots ({\mathbf 1} - {\zeta}_{n})  \label{multln}\\
      & = & \sum_{k=1}^{n} \, (-1)^{k-1} D_{[k-1]} \, , \label{addln}
\end{eqnarray}
where $\mathbf 1$ denotes the identity permutation and ${\zeta}_{k}$ the descending $k$-cycle $(k \, \ldots \, 2 \, 1)$.
Note that in \eqref{multln} the products $\sigma \tau$ of permutations ${\sigma}, {\tau} \in {\mathfrak S}_{n}$ are to be read from right
to left: first $\tau$ and then $\sigma$.
It is well known that the elements $D_{X}$ span a subalgebra of rank $2^{n-1}$ of the group algebra ${\mathbb Q}{\mathfrak S}_{n}$,
called the {\em Solomon descent algebra} and denoted by ${\mathcal D}_{n}$  (e.g., see \cite{Scho} and cf. \cite[Chapter 9]{Reut}).
By \eqref{addln} it follows that $l_{n}$ lies in ${\mathcal D}_{n}$.

\smallskip

We also define $l^{*}_{n}$ to be the element of  $K{\mathfrak S}_{n}$ such that
\begin{equation}
(x_{1} x_{2} \ldots x_{n}) \cdot l^{*}_{n} \: = \: l^{*}(x_{1} x_{2} \ldots x_{n}) \, , \label{defln*}
\end{equation}
where $x_1 x_{2} \ldots x_n$ is a word in $n$ distinct letters and obtain the following result.

\medskip

\begin{lemma}\label{lnln*symm}
\begin{description}
\item[\em{(i)}]
Suppose that $\displaystyle l_{n} = \sum_{{{\sigma} \in {\mathfrak S}_{n}}} {\alpha}_{\sigma} {\sigma}$ and
$\: \displaystyle l^{*}_{n} = \sum_{{{\sigma} \in {\mathfrak S}_{n}}} {\beta}_{\sigma} {\sigma}$.
Then $\displaystyle {\beta}_{\sigma} = {\alpha}_{{\sigma}^{-1}}$.
\item[\em{(ii)}]
$ \quad l^{*}_{n} \: = \: ({\mathbf 1} - {\zeta}_{n}^{-1}) \cdots ({\mathbf 1} - {\zeta}_{3}^{-1})({\mathbf 1} - {\zeta}_{2}^{-1}) \, . $
\end{description}
\end{lemma}

\begin{proof}
(i) The coefficients ${\beta}_{\sigma}$ and ${\alpha}_{{\sigma}^{-1}}$ are related via the canonical scalar product in $K {\langle A \rangle}$
defined by \eqref{scalarproduct} in the following way :
\begin{eqnarray*}
{\beta}_{\sigma} & = &  \bigl( (x_{1} x_{2} \ldots x_{n}) \, \cdot \, l^{*}_{n}, \quad x_{{\sigma}(1)} x_{{\sigma}(2)} \ldots x_{{\sigma}(n)} \bigr)
\qquad [\mbox{by definition}] \\
& = & \bigl( l^{*}(x_{1} x_{2} \ldots x_{n}), \quad x_{{\sigma}(1)} x_{{\sigma}(2)}\ldots x_{{\sigma}(n)} \bigr) \qquad [\mbox{ by \eqref{defln*}}] \\
& = & \bigl( l( x_{{\sigma}(1)} x_{{\sigma}(2)} \ldots x_{{\sigma}(n)}), \quad  x_{1} x_{2} \ldots x_{n} \bigr) \qquad [\mbox{by \eqref{defl*}}].
\end{eqnarray*}
Setting new variables $y_{i} = x_{{\sigma}(i)}$ for $i = 1, 2, \ldots , n$ we have $y_{{{\sigma}^{-1}}(i)} = x_{i}$, so we obtain
\begin{eqnarray*}
{\beta}_{\sigma} & = &
\bigl( l( y_{1} y_{2} \ldots y_{n}), \quad y_{{{\sigma}^{-1}}(1)} y_{{{\sigma}^{-1}}(2)} \ldots y_{{{\sigma}^{-1}}(n)} \bigr) \\
& = & \bigl( ( y_{1} y_{2} \ldots y_{n}) \, \cdot \, l_{n}, \quad y_{{{\sigma}^{-1}}(1)} y_{{{\sigma}^{-1}}(2)} \ldots y_{{{\sigma}^{-1}}(n)} \bigr)
\qquad [\mbox{ by \eqref{defln}}] \\
& = & {\alpha}_{{\sigma}^{-1}} \qquad [\mbox{by definition}].
\end{eqnarray*}

(ii) Let ${\sigma}_{1}, \, {\sigma}_{2}, \ldots , {\sigma}_{k}$ be arbitrary elements of ${\mathfrak S}_{n}$.
By induction on $k$ it is straightforward to check that if
$\displaystyle ({\mathbf 1} -  {\sigma}_{1})({\mathbf 1} -  {\sigma}_{2}) \cdots ({\mathbf 1} -  {\sigma}_{k}) =
\sum_{{\sigma} \in {\mathfrak S}_{n}} {\beta}_{\sigma} {\sigma}$ then
$\displaystyle ({\mathbf 1} - {\sigma}_{k}^{-1}) ({\mathbf 1} - {\sigma}_{k-1}^{-1}) \cdots ({\mathbf 1} - {\sigma}_{1}^{-1}) =
\sum_{{\sigma} \in {\mathfrak S}_{n}} {\beta}_{\sigma} \, {\sigma}^{-1}$.
Our result follows from this property and part (i).
\end{proof}

\bigskip

We carry on with some preliminaries on set partitions and tabloids.
A {\em composition} $\lambda$ of a positive integer $n$ into $r$ positive parts, written $\lambda \models n$,
is an ordered sequence $({\lambda}_{1}, {\lambda}_{2}, \ldots , {\lambda}_{r})$ of positive integers
such that $\sum_{j=1}^{r} {\lambda}_{j} = n$. If in addition ${\lambda}_{1} \geq {\lambda}_{2} \geq \cdots \geq {\lambda}_{r}$
then $\lambda$ is called an ({\em integer) partition} of $n$, written $\lambda  \vdash  n$.
An {\em ordered set partition} (or {\em set composition}) $P$ of $[n]$ into $r$ parts is an ordered $r$-tuple
$P = ( I_{1}, I_{2}, \ldots , I_{r} )$ of $r$ pairwise disjoint non empty subsets $I_{k}$ of $[n]$ (called {\em blocks})
whose union is $[n]$. If we forget the ordering of the blocks and consider just the collection $\pi = \{ I_{1} , I_{2} , \ldots , I_{r} \}$
we obtain an {\em (unordered) set partition} $\pi$ of $[n]$ into $r$ parts.
The {\em type} of $P$ is the composition ${\lambda}(P) = ( |I_{1}|, |I_{2}|, \ldots , |I_{r}|)$ of $n$ and its {\em length} $l(P)$ is the number
of blocks $r$. We let ${\Pi}_{n}^{r}$ (respectively ${\Delta}_{n}^{r}$) denote the set of ordered (respectively unordered) partitions
of $[n]$ with $r$ blocks and ${\Pi}_{n}$ (respectively ${\Delta}_{n}$) be the set of all ordered (respectively unordered) set partitions
of $[n]$. Let also ${\mathcal T}_{n}^{\lambda}$ be the set of ordered partitions of given type
$\lambda = ({\lambda}_{1}, {\lambda}_{2}, \ldots , {\lambda}_{r})$,
where $\lambda \models n$. This is nothing but the set of {\em $\lambda$-tabloids}.
(Note that $\lambda$-tabloids are usually defined for $\lambda  \vdash  n$ as row equivalence classes of {\em Young tableaux of shape $\lambda$}
(see \cite[Def. 2.1.4]{Saga}), but the same can be done in general for $\lambda \models n$ since the definition of a tableau is extended to
compositions in the obvious way (cf. \cite[p. 67]{Saga}).)

It is well known (see \cite[\S 6.1]{Grah + Knut + Pata}) that $|{\Delta}_{n}^{r}|$ is equal
to the {\em Stirling number of the second kind}, denoted by $\{ {n \atop r } \}$, which can be computed by the sum
$\frac{1}{r!} \, \sum_{i=0}^{r} \, (-1)^{i} \binom{r}{i} {(r - i)}^{n}$ and $|{\Delta}_{n}|$ is equal to
the $n$-th {\em Bell number} $B_{n}$ which can recursively be defined as $B_{0} = 1$ and $B_{n} = \sum_{i=0}^{n-1} \binom{n-1}{i} \, B_{i}$.
Then $|{\Pi}_{n}^{r}|$ is equal to $r! \, \{ {n \atop r } \}$, since each element in
${\Delta}_{n}^{r}$ yields $r!$ distinct elements in ${\Pi}_{n}^{r}$ by permuting its blocks.
On the other hand, clearly $|{\mathcal T}_{n}^{\lambda}|$ is equal to the multinomial coefficient
$\displaystyle \binom{n}{\lambda} := \binom{n}{{\lambda}_{1}, \, {\lambda}_{2}, \, \ldots , \, {\lambda}_{r}}
= \frac{n!}{{\lambda}_{1}! {\lambda}_{2}! \cdots {\lambda}_{r}!}$.

Following Sagan in \cite[Def. 5.5.7]{Saga} each $\pi \in {\Delta}_{n}^{r}$ may be written uniquely in the form of a special ordered
partition, called the {\em tabloid form} of $\pi$, which is defined as
\begin{equation}
\mathcal P \: = \: I_{1} / I_{2} / \ldots / I_{r} \, , \label{tablunpart}
\end{equation}
where numbers in each block are written in the natural increasing order; blocks are listed in weakly decreasing order of size
and furthermore blocks of equal size are arranged in increasing order of their minimal elements.
The {\em type} of $\pi$ is then the integer partition
${\lambda}({\pi}) = ( |I_{1}|, |I_{2}|, \ldots , |I_{r}|)$ of $n$.
For example the partition $\pi \: = \: \{ \{1, 2 \}, \, \{4, 7, 8 \}, \, \{3, 5, 6 \} \}$  of $[8]$ is written in tabloid form as
$\mathcal P \: = \: 3, 5, 6 \, / \, 4, 7, 8 \, / \, 1, 2$ and is of type $(3,3,2)$.

Let $P, Q \in {\Pi}_{n}$. We say that $P$ {\em refines} $Q$ and write $P \preceq Q$ if each block of $P$ is a subset of some block of $Q$.
The relation $\preceq$ on ${\Pi}_{n}$ is reflexive and transitive. We have $P \preceq Q$ and $Q \preceq P$ if and only if $Q$ may be obtained
by rearranging the blocks of $P$. In this case we write $P \simeq Q$ and obtain an equivalence relation $\simeq$ in ${\Pi}_{n}$. The set
of its equivalence classes is then clearly identified with the set ${\Delta}_{n}$ of unordered set partitions of $[n]$
which inherits the refinement order $\preceq$ from ${\Pi}_{n}$. The $n$-block partition $1 \, / 2 \, / \, \ldots \, / \, n$ and
the $1$-block partition $1, 2, \ldots , n \, / \,$ appear respectively at the bottom and at the top of the corresponding {\em Hasse diagram}
of the partially ordered set $({\Delta}_{n}, \, \preceq )$.

The symmetric group ${\mathfrak S}_{n}$ acts naturally from the left on the sets
${\Delta}_{n}^{r}$, ${\Pi}_{n}^{r}$ and ${\mathcal T}_{n}^{\lambda}$ defined above,
simply by permuting the entries in the blocks of a set partition or a tabloid.
For a commutative ring $K$ with unity let $D_n^r$, $P_n^r$ and $T_n^{\lambda}$ be the $K$-modules freely generated by the sets
${\Delta}_{n}^{r}$, ${\Pi}_{n}^{r}$ and ${\mathcal T}_{n}^{\lambda}$, respectively.
Extending the permutation action of ${\mathfrak S}_{n}$ linearly to $D_n^r$, $P_n^r$ and $T_n^{\lambda}$ these become
left permutation $K{\mathfrak S}_{n}$-modules.

\medskip

In Problem \ref{suppfreeLieZm} we search for words $w$ with $|w| = n$ and $|alph(w)| = r$ that vanish under $l^{*}$.
By Lemma \ref{l*2alph}\,(ii) if $B_1$ and $B_2$ are sub-alphabets of cardinality $r$ of $A$ then $\ker l_{B_{1}}^{*}$ is identified with
$\ker l_{B_{2}}^{*}$. Therefore, without loss of generality, we may fix a sub-alphabet $B = \{ a_{1}, a_{2}, \ldots a_{r} \}$ of $A$
with the natural total order $a_{1} < a_{2} < \cdots a_{r}$ and consider the set ${\mathcal W}_n^r$ of all words $w$ with $|w| = n$ and $alph(w) = B$.
The mapping $w \mapsto \{ w \}$ from ${\mathcal W}_n^r$ to ${\Pi}_n^r$ that sends a
word $w = x_{1} x_{2} \ldots x_{n} \, \in \, {\mathcal W}_n^r$ to the ordered
partition $\{w\} = (I_{1}(w), I_{2}(w), \ldots , I_{r}(w))$, where for each $k \in [r]$ the set $I_{k}(w)$ consists
of the positions in $[n]$ where the letter $a_k$ occurs in $w$, is clearly a bijection.
Moreover, if $\lambda = ({\lambda}_{1}, {\lambda}_{2}, \ldots , {\lambda}_{r}) \models n$
we can also consider the set ${\mathcal W}_{n}^{\lambda}$ of all words $w$ of multi-degree
$({\lambda}_{1}, {\lambda}_{2}, \ldots , {\lambda}_{r})$ in $B$, \textit{i.e.}, $|w|_{a_{k}} = {\lambda}_{k}$ for each
$k \in [r]$. The ordered partition $\{w\}$ is then a $\lambda$-tabloid, hence the restriction of the map  $w \mapsto \{ w \}$ to
${\mathcal W}_{n}^{\lambda}$ is a bijection between the set of words ${\mathcal W}_{n}^{\lambda}$ and the set
${\mathcal T}_{n}^{\lambda}$ of $\lambda$-tabloids. For example, for $\lambda = (3, 3, 1, 1)$ and $B = \{ a, b, c, d \}$, the word $w = aacbdbba$
is represented by the tabloid $1, 2, 8 \, / \, 4, 6, 7 \, / \, 3 \, / \, 5$.
By changing our initial order on $B$ to another one which makes $\lambda = (|w|_{a_{1}}, |w|_{a_{2}}, \ldots , |w|_{a_{r}})$ an integer
partition of $n$ we may, without loss of generality, assume that $\lambda \vdash n$ and consider only the sets ${\mathcal W}_{n}^{\lambda}$
and ${\mathcal T}_{n}^{\lambda}$ for $\lambda \vdash n$.

\smallskip

Let $W_n^r$  be the $K$-span of ${\mathcal W}_n^r$ in $K \langle B \rangle$
and $W_{n}^{\lambda}$ be the set of all polynomials on $B$ over $K$ of multi-degree
$({\lambda}_{1}, {\lambda}_{2}, \ldots , {\lambda}_{r})$.
The symmetric group ${\mathfrak S}_{n}$ acts on the sets ${\mathcal W}_n^r$ and ${\mathcal W}_{n}^{\lambda}$
from the right by place permutations.
Extending these actions naturally to $W_n^r$ and $W_{n}^{\lambda}$ the latter become right permutation
modules for ${\mathfrak S}_{n}$.
We compare the left action of ${\mathfrak S}_{n}$ on $P_{n}^{r}$ and $T_{n}^{\lambda}$ with its right action on $W_{n}^{r}$ and $W_{n}^{\lambda}$, respectively by the following results.

\begin{lemma}\label{symm{w}}
Let $w$ be a word of length $n$ and $\sigma$ be an arbitrary permutation of the symmetric group
${\mathfrak S}_{n}$ on $n$ letters. Then
\[ {\sigma} \cdot \{ w \} \: = \: \{w \cdot {\sigma}^{-1} \} \, . \]
\end{lemma}

\smallskip

\begin{proof}
Since $\{w\} = (I_{1}(w), I_{2}(w), \ldots , I_{r}(w))$ it suffices to show that $I_{k}(w \cdot {\sigma}^{-1}) = \sigma \cdot I_{k}(w)$,
for each $k \in [r]$.
By definition $I_{k}(w) = \{ i \in [n] \, : \, x_{i} = a_{k} \}$, hence clearly
$\sigma \cdot I_{k}(w) = \{ {\sigma}(i) \in [n] \, : \, x_{i} = a_{k} \} =
\{ j \in [n] \, : \, x_{{\sigma}^{-1}(j)} = a_{k} \} = I_{k}(w \cdot {\sigma}^{-1})$.
\end{proof}

We extend by linearity the bijective mapping $w \mapsto \{ w \}$ from ${\mathcal W}_{n}^{r}$ to
${\Pi}_{n}^{r}$ to a $K$-module isomorphism ${f}$ from $W_{n}^{r}$ to $P_{n}^{r}$
defined as ${f}(\sum_{w} k_{w} w) = \sum_{w} k_{w} \{ w \}$, for a typical element $\sum_{w} k_{w} w \in W_{n}^{r}$.

\begin{theorem}\label{wln*ln{w}}
The right action of the symmetric group ${\mathfrak S}_{n}$ on the polynomials $W_{n}^{r}$
(respectively $W_{n}^{\lambda}$, where $\lambda \vdash n$) is equivalent to the left action of ${\mathfrak S}_{n}$ on the $K$-space
of ordered partitions $P_{n}^{r}$ of $[n]$ into $r$ parts (respectively on $\lambda$-tabloids $T_{n}^{\lambda}$).
Moreover,
\[ {f} (w \cdot l_{n}^{*}) \: = \: l_{n} \cdot \{ w \} \, . \]
\end{theorem}

\begin{proof}
If we define ${\sigma} \circ w = w \cdot {\sigma}^{-1}$ for $\sigma \in {\mathfrak S}_{n}$ and
$w \in {\mathcal W}_{n}^{r}$, we have also a left action of
${\mathfrak S}_{n}$ on ${\mathcal W}_{n}^{r}$ and consequently on $W_{n}^{r}$.
In view of Lemma \ref{symm{w}} one can show that
${f} (\sigma \circ P) = \sigma \cdot {f} (P)$, for a polynomial $P \in W_{n}^{\lambda}$.
As a result the corresponding left $K{\mathfrak S}_{n}$-modules
$W_{n}^{r}$ and $P_{n}^{r}$ are isomorphic. The restriction of $f$ to $W_{n}^{\lambda}$ is an isomorphism of $W_{n}^{\lambda}$
onto $T_{n}^{\lambda}$. The second part of the theorem follows directly from Lemma \ref{symm{w}} and Lemma \ref{lnln*symm}\,(i).
\end{proof}

Consider the involution ${\tau}_{n}$ of the symmetric group ${\mathfrak S}_{n}$ written as
\begin{equation}
{\tau}_{n}  \: =  \: \prod_{i=1}^{{\left\lfloor \frac{n}{2} \right\rfloor}} \, (i, \, n-i+1) \, . \label{invo}
\end{equation}

\begin{lemma}\label{tildewtn{w}}
The right action of $l_{n}^{*}$ on $\tilde{w}$ is equivalent to the left action of $l_n$ to the tabloid
${\tau}_{n} \cdot \{ w \}$, \textit{i.e.},
\[ {f}({\tilde{w}} \cdot l_{n}^{*})  \: =  \: l_{n} \cdot ( {\tau}_{n} \cdot \{ w \} ) \, . \]
\end{lemma}

\begin{proof}
Clearly $w \cdot {\tau}_{n} = \tilde{w}$ so if we apply Lemma \ref{symm{w}} for $\sigma = {\tau}_{n}$ we obtain
${\tau}_{n} \cdot \{ w \}  =  \{ \tilde{w} \}$. The result then follows from Theorem \ref{wln*ln{w}}.
\end{proof}

A major implication of the isomorphism established in Theorem \ref{wln*ln{w}} is the equivalence
\begin{equation}
w \cdot l^{*}_{n} \: = \: 0  \quad \Longleftrightarrow  \quad l_{n} \cdot \{w\} \: = \: 0, \label{lnequivln*}
\end{equation}
so Problem \ref{suppfreeLieZm} takes the following form.

\begin{problem}\label{tabloidZm}
Let $m$ be a positive integer with $m \neq 1$. Find all unordered set partitions of $[n]$, written in tabloid form
${\mathcal P} = I_{1} / I_{2} / \ldots / I_{r}$, with the property that $\mathcal P$ satisfies
\begin{equation}
l_n \, \cdot \, {\mathcal P} \: = \: 0 \, , \label{tableq}
\end{equation}
where $l_n$ is the left normed Lie bracketing viewed as an element of the group ring ${\mathbb Z}_{m}{\mathfrak S}_{n}$.
\end{problem}

A few points need to be clarified here.
\begin{enumerate}
\item
Theorem \ref{alphbound} imposes the restriction $r = l({\mathcal P}) \leq {\lceil n/2 \rceil}$ in \eqref{tableq}.
Fix such a length $r$. If a specific ordered partition $P = (I_{1}, I_{2}, \ldots , I_{r})$ satisfies (\ref{tableq}) then by
Lemma \ref{l*2alph}\,(ii) and the equivalence \eqref{lnequivln*} all $r!$ ordered partitions formed by permutations of its blocks will also
be solutions of (\ref{tableq}). Hence it suffices to consider only ordered partitions in tabloid form \eqref{tablunpart}, \textit{i.e.}, unordered partitions, of length $r$ where $2 \leq r \leq {\lceil n/2 \rceil}$ since the $1$-block partition is always a trivial solution
of \eqref{tableq}.
\item
If $B_{1}$  and $B_{2}$ are sub-alphabets of $A$ with $B_{1} \supset  B_{2}$, $\phi$
is a literal surjective morphism from $B_{1}^{*}$ onto $B_{2}^{*}$ and there exists a word $w \in B_{1}^{*}$ such that
$l_n \cdot \{ w \} = 0$  then Lemma \ref{l*2alph}\,(i) and the equivalence \eqref{lnequivln*} imply that also $l_n \cdot \{ {\phi}(w) \} = 0$.
Consequently, if $P, Q \in {\Pi}_{n}$, $P$ is a solution of (\ref{tableq}) and $P \preceq Q$ then
$Q$ will also satisfy (\ref{tableq}).
\item
In view of the two previous remarks to solve Problem \ref{tabloidZm} we must determine the smallest set of unordered partitions $\mathcal P$ in
tabloid form \eqref{tablunpart} with $r$ parts that satisfy (\ref{tableq}) and generate all solutions of (\ref{tableq}) in the sense that
every other solution will be a partition $\mathcal Q$ refined by $\mathcal P$ and moreover $\mathcal P$ can not itself be refined by some other
solution $\mathcal R$ of (\ref{tableq}) with more than $r$ parts. We search for such a set of minimal solutions starting from partitions of
length ${\lceil n/2 \rceil}$ and moving up on in the Hasse diagram of the partially ordered set $({\Delta}_{n}, \, \preceq )$.

For example, when $n = 5$ and $m = 2$ one can show that the solutions with $3$ parts are the tabloids of the form
$1, 3, 5 \, / \, 2 \, / \, 4$ and $1, 5 \, / \, 2, 4 \, / \, 3$
and every solution with $2$ parts is generated from those. On the other hand, for $n = 7$ and $m = 2$ we have
the solution $1, 3, 6 \, / \, 2, 5, 7 \, / \, 4$ with $3$ parts which can not be refined by any solution
with $4$ parts; the latter follows directly from Example \ref{example}.
\end{enumerate}

\smallskip

In characteristic zero Theorem \ref{DuchThib} and the equivalence \eqref{lnequivln*} yield the following result,
which - according to our knowledge - has not been traced in the literature.

\begin{proposition}\label{kerlnZ}
Let $\lambda \vdash n$. The $\lambda$-tabloids that satisfy the equation
\begin{equation*}
l_n \, \cdot \, t \: = \: 0 \, ,
\end{equation*}
where $l_n$ is the multi-linear left normed Lie bracketing of the free Lie ring are either the $1$-block partition
$1 , 2 , \ldots , n \, / \,$ or the tabloids $t$ which, when viewed as partitions of $[n]$, are refined by the tabloid
\begin{equation}
 1, n \, / \, 2, n-1 \, / \, 3, n-2 \, / \ldots / \, k, k+1 \, , \label{evenpal}
\end{equation}
for $n = 2k$.
\end{proposition}

\smallskip

We also present the equivalent of Conjecture \ref{conjtwinantiZ} stated in $\lambda$-tabloid form.

\begin{conjecture}\label{conjtabloidZ}
Let $\lambda \vdash n$ and $t_1$ and $t_2$ be $\lambda$-tabloids with both $l_{n} \, \cdot \, t_1$ and $l_{n} \, \cdot \, t_2$ different
from zero.
\begin{description}
 \item[\em{(i)}]
$l_{n} \, \cdot \, t_1 \: = \: l_{n} \, \cdot \, t_2$ \, if and only if \, $t_1 = t_2$ or $n$ is odd and \, $t_1 = {{\tau}_{n}} \, \cdot \, t_2$.
 \item[\em{(ii)}]
$l_{n} \, \cdot \, t_1 \: = \: - \,\, l_{n} \, \cdot \, t_2$ \, if and only if \, $n$ is even and \, $t_1 = {{\tau}_{n}} \, \cdot \, t_2$.
\end{description}
\end{conjecture}

\medskip

We conclude this section with a few remarks on the case $m=2$. It is easy to show by induction that
palindromes of length $n > 1$, where $n$ might be even or odd, lie in the kernel of $l^{*}$.
The latter correspond of course to tabloids of the form
\begin{equation}
t = 1, n \, / \, 2, n-1 \, / \, 3, n-2 \, / \ldots / \, r, r + 2 \, / \, r + 1 \, ,  \label{oddpal}
\end{equation}
where $n = 2r + 1$. By an easy induction on $r$ we also get the hook-shaped solution
\begin{equation}
1, 3, 5, \ldots , 2r + 1 \, / \, 2 \, / \, 4 \, / \, 6 \, / \ldots / \, 2r \, . \label{oddhook}
\end{equation}
We conjecture that (\ref{oddpal}) and (\ref{oddhook}) are the only solutions of (\ref{tableq}) with
$r + 1$ parts when $n=2r + 1$.
For $n=2r$, except from tabloids of the form (\ref{evenpal}) that correspond to palindromes of even length,
there exist other solutions with $r$ parts.
For example for $n = 6$ we have also the solutions
$1, 3, 5 \, / \, 2, 6 \, / \, 4$; $1, 3, 6 \, / \, 2, 5 \, / \, 4$; $1, 4, 6 \, / \, 2, 5 \, / \, 3$ and
$2, 4, 6 \, / \, 1, 5 \, / \, 3$.
Nevertheless, we conjecture that the only solution of
type $(2, 2, \ldots , 2)$ with $r$ parts is of the form (\ref{evenpal}).

\bigskip


\section{Pascal descent polynomials}\label{pascal}

In this section we restrict ourselves to words of length $n$ where only two letters occur and consider the corresponding set
${\Delta}_{n}^{2}$ of all unordered set partitions of $[n]$ with $2$ parts. It is known
\cite[p. 258]{Grah + Knut + Pata} that $|{\Delta}_{n}^{2}| = \{ {n \atop 2 } \} = 2^{n-1}  - 1$.
Such a partition is written in tabloid form $J \, / \, I$, so it is uniquely determined by the subset
$I = \{ i_1 , i_2 , \ldots , i_s \}$ of $[n]$ appearing in its second block and will be denoted accordingly by $\overline{I}$;
when $I = \{ i \}$ for brevity we will write $\overline{i}$.

We define a mapping $\psi$ from the set ${\Pi}_{n}^{2}$ of ordered partitions with $2$ parts to
the polynomial algebra $K[x_1 , x_2 , \ldots , x_n ]$ in $n$ commuting variables $x_1 , x_2 , \ldots , x_n$
that sends $\overline{I}$ - viewed as an ordered partition - to the monomial $x_{I} = x_{i_1} x_{i_2} \cdots x_{i_s}$. In the extreme case
where $I = \emptyset$ we may set $x_{I} = 1$.
The mapping $\psi$ is extended by linearity to a $K{\mathfrak S}_{n}$-module isomorphism, also denoted by $\psi$, from $P_{n}^{2}$ to
$K[x_1 , x_2 , \ldots , x_n ]$, under the natural place permutation actions of the symmetric group ${\mathfrak S}_{n}$ in each case.

Let $h_{n}(I)$ be the image under $\psi$ of the element $l_n \cdot \overline{I}$; for simplicity we write $h_{n}(i)$ instead of $h_{n}(\{i\})$.
Let us clarify this definition a bit more. Suppose that $l_{n} = \sum_{{{\sigma} \in {\mathfrak S}_{n}}} {\alpha}_{\sigma} {\sigma}$.
Then we obtain
\begin{equation}
 h_{n}(I) \: = \: \sum_{{{\sigma} \in {\mathfrak S}_{n}}} {\alpha}_{\sigma} {\psi}({\sigma} \cdot \overline{I})
          \: = \: \sum_{{{\sigma} \in {\mathfrak S}_{n}}} {\alpha}_{\sigma} {\psi}(\overline{{\sigma}(I)})
         \: = \: \sum_{{{\sigma} \in {\mathfrak S}_{n}}} {\alpha}_{\sigma} x_{{\sigma}(I)}. \label{defhn}
\end{equation}
The polynomial $h_{n}(I)$ is homogeneous of total degree $s$ and {\em multi-linear}, \textit{i.e.}, the degree of each of its variables is at most $1$.

Let now $N$ be a positive integer with $N > n$ and set $[n+1,N] = [N] \setminus [n]$. It will be useful for us to extend the definition of
$h_{n}(I)$ to subsets of $[N]$. We do this using the last equality of \eqref{defhn}, \textit{i.e.},
$h_{n}(I) = \sum_{{{\sigma} \in {\mathfrak S}_{n}}} {\alpha}_{\sigma} x_{{\sigma}(I)}$ and by identifying ${\mathfrak S}_{n}$ with the
set of permutations $\sigma \in {\mathfrak S}_{N}$ that leave point-wise invariant the subset $[n+1,N]$ of $[N]$.
We obtain the following technical result.

\begin{lemma}\label{exthn}
Let $n, N$ be positive integers with $n < N$ and $I \subseteq [N]$ such that $J = I \cap [n+1,N] \neq \emptyset$.
Then
\[ h_{n}(I) \: = \: h_{n}(I \setminus J) \, x_{J} \, . \]
\end{lemma}

\smallskip

\begin{proof}
In view of the aforementioned identification we let $\sigma \in {\mathfrak S}_{n}$. Then ${\sigma}(I)$ is the disjoint union of
${\sigma}(I \setminus J)$ and $J$ since ${\sigma}(J) = J$, so that \eqref{defhn} yields
\begin{equation*}
h_{n}(I)  \: = \: \sum_{{\sigma} \in {\mathfrak S}_{n}} {\alpha}_{\sigma} \, x_{{\sigma}(I)}
         \: = \: \sum_{{\sigma} \in {\mathfrak S}_{n}} {\alpha}_{\sigma} \, x_{{\sigma}(I \setminus J)} \, x_{J}
        \:  = \: \big [ \sum_{{\sigma} \in {\mathfrak S}_{n}} {\alpha}_{\sigma} \, x_{{\sigma}(I \setminus J)} \big ] \, x_{J}
          \: = \: h_{n}(I \setminus J) \, x_{J} \, .
\end{equation*}
as required.
\end{proof}

\begin{lemma}\label{defrechn}
Let $I \subseteq [n]$. Then the polynomial $h_{n}(I)$ is recursively defined as
\[ h_{n}(I)  = \begin{cases}
   \displaystyle (-1)^{i-1} {{n  - 1 } \choose {i - 1}} \, \{ x_{1} - x_{2} \},  &\text{if \, $I =\{i\}$;} \\
   h_{n-1}(I) \; - \; h_{n-1}({\zeta}_{n} I),  &\text{otherwise,}
               \end{cases}  \]
where ${\zeta}_{n}$ denotes the descending cycle $(n \, \ldots \, 2 \, 1)$.
\end{lemma}

\smallskip

\begin{proof}
For $I = \{ i \}$ Theorem \ref{wln*ln{w}} and Lemma \ref{l*1b} yield
\begin{eqnarray*}
h_{n}(i) \: = \:
{\psi}(l_{n} \cdot {\overline{i}} \, ) & = & {\psi} ({f} \big( ( a^{i-1}ba^{n-i} ) \cdot l_{n}^{*} \big))
      \: = \: {\psi} \bigl( {f} ( l^{*}( a^{i-1}ba^{n-i} ) ) \bigr) \\
       & = & ({\psi}  {f}) ( (-1)^{i-1} {{n  - 1 } \choose {i - 1}} \{ ba^{n - 1 } - aba^{n - 2} \} ) \\
       & = & (-1)^{i-1} \binom{n - 1 }{i - 1} \{ ({\psi}  {f}) (ba^{n-1}) - ({\psi}  {f}) (aba^{n-2}) \} \\
       & = & (-1)^{i-1} {{n  - 1 } \choose {i - 1}} \{ ({\psi} ({\overline{1}}) - {\psi} ({\overline{2}}) \} =
             (-1)^{i-1} {{n  - 1 } \choose {i - 1}} \{ x_1  - x_2 \} .
\end{eqnarray*}
When $|I| > 1$ we argue by induction on $n$, where $n \geq 2$.
The case $n=2$ follows trivially.
Using the multiplicative formula (\ref{multln}) for $l_n$ the induction step yields
\begin{eqnarray*}
h_{n}(I) \: = \: {\psi}(l_{n} \cdot \overline{I}) & = & {\psi} \big [ l_{n-1} (1 - {\zeta}_{n} ) \cdot \overline{I} \big ] \\
       & = & {\psi} \big [ l_{n-1} \cdot \overline{I} \; - \; l_{n-1} \cdot ( {\zeta}_{n} \overline{I} ) \big ] \\
       & = & h_{n-1}(I) - h_{n-1} \big ( {\zeta}_{n} I \big ) \, .
\end{eqnarray*}
Note that if $n \in I$ or $1 \in I$ we get $h_{n-1}(I) = h_{n-1}(I \setminus \{n\})\,x_{n}$ and
$h_{n-1}({\zeta}_{n} I) = h_{n-1}({\zeta}_{n} I \setminus \{n\})\,x_{n}$, respectively, by Lemma \ref{exthn}.
\end{proof}

\medskip

\begin{proposition}\label{defrecpn}
Let $I$ be a subset of $[n]$ of cardinality $s$. The binomial $x_{1} - x_{2}$ divides $h_{n}(I)$.
The corresponding quotient, which we denote by $p_{n}(I)$, is a multi-linear polynomial of total degree $s - 1$
given by the recursive formula
\begin{equation*}
p_{n}(I)  = \begin{cases}
   \displaystyle (-1)^{i-1} {{n  - 1 } \choose {i - 1}},  &\text{if \, $I =\{i\}$;} \\
   p_{n-1}(I) \; - \; p_{n-1}({\zeta}_{n} I),  &\text{otherwise,} \label{pnrec}
               \end{cases}
\end{equation*}
where whenever $I'$ is a subset of $[n]$ with $n \in I'$ we set $p_{n-1}(I') = p_{n-1}(I' \setminus \{n\}) \, x_{n}$.
\end{proposition}

\begin{proof}
When $s=1$ Lemma \ref{defrechn} immediately yields $p_{n}(i) = (-1)^{i-1} \binom{n-1}{i-1}$ so $p_{n}(I)$ is a generalization of the signed binomial
coefficient.
For $s  > 1$ we argue by induction on $n$. The induction step from $n-1$ to $n$ goes as follows.
We set $J = I \cap \{n\}$ and $K = {\zeta}_{n} \, I \cap \{n\}$, where clearly
$x_{J} = 1$ and $x_{K} = 1$ when $n \notin I$ and $1 \notin I$, respectively.
By Lemma \ref{defrechn} $h_{n}(I) = h_{n-1}(I) \: - \:  h_{n-1}({\zeta}_{n} I)$, so by Lemma \ref{exthn} we obtain
$h_{n}(I) = h_{n-1}(I \setminus J) \, x_{J} \: \, - \: h_{n-1}({\zeta}_{n} I \setminus K) \, x_{K}$. By our induction
hypothesis $h_{n-1}(I \setminus J) = (x_{1} - x_{2}) \, p_{n-1}(I \setminus J)$ and
$h_{n-1}({\zeta}_{n} I \setminus K) = (x_{1} - x_{2}) \, p_{n-1}({\zeta}_{n} I \setminus K)$.
Our result then follows if we set $p_{n}(I) =  p_{n-1}(I \setminus J) \, x_{J} \: - \: p_{n-1}({\zeta}_{n} I \setminus K) \, x_{K}$.
\end{proof}

We give the name {\em Pascal descent polynomial} to $p_{n}(I)$ since it yields a signed binomial coefficient when $I$ is a singleton set and
it originates from the additive formula \eqref{addln} which relates to the descent sums $D_{[k-1]}, \, k \in [n]$.

\smallskip

By Lemma \ref{tildewtn{w}}, the analogue of Lemma \ref{l*reversal} for Pascal descent polynomials is given by the equation
\begin{equation}
p_{n}({{\tau}_{n}}(I))  \: = \: (-1)^{n+1} \, p_{n}(I),  \label{pnreversal}
\end{equation}
where ${\tau}_{n}$ is the involution of ${\mathfrak S}_{n}$ defined in \eqref{invo}.
In particular, if $n$ is even and ${{\tau}_{n}}(I) = I$ then $p_{n}(I) = 0$.

\bigskip

The following result yields a recursive decomposition of the polynomial $p_{n}(I)$.

\smallskip

\begin{theorem}\label{recformulapn}
Let $I = \{ a_{1}, a_{2}, \ldots , a_{d} \}$ be a proper subset of $[n]$ of cardinality $d \geq 2$. Set $k = a_{1} - 1$ and $l = n - a_{d}$. Then
\begin{equation*}
p_{n}(I)  \quad = \quad (-1)^{k+1} \sum_{i=0}^{l} \, \binom{k+i}{i} \, p_{\scriptstyle n-k-1-i}(D) \, \cdot \, x_{\scriptstyle n-k-i}
          \quad + \quad  \sum_{j=0}^{k} \, (-1)^{j} \binom{l+j}{j} \, p_{\scriptstyle n-l-1-j}(D_{j}) \, \cdot \, x_{\scriptstyle n-l-j},
\end{equation*}
where $D = \{a_{2} - a_{1}, \, a_{3} - a_{1}, \, \ldots , \, a_{d} - a_{1}\}$ and
$D_{j} = \{a_{1} - j, \, a_{2} - j, \, \ldots , \, a_{d-1} - j\}$, for each $j = 0, \ldots , k$.
\end{theorem}

\smallskip

\begin{proof}
We apply Proposition \ref{defrecpn}. The case where $a_{1} = 1$ and $a_{d} = n$ follows trivially.
The cases $a_{1} = 1, \: a_{d} < n$ and  $a_{1} > 1, \: a_{d} = n$ follow by an immediate induction on $n$.
If $1 < a_{1}$ and $a_{d} < n$ the induction on $n$ is a bit more tedious. For $n = 3$ the result is true; one can check it easily for
$I = \{ 1, 2 \}$ and $I = \{ 2, 3 \}$. Suppose that it holds for $n = m$. We will show that it also holds for $n = m + 1$.
Setting $k = a_{1} - 1$ and $l = m + 1 - a_{d}$, Proposition \ref{defrecpn} and our induction hypothesis yield
\begin{eqnarray*}
p_{m+1}(I) & = &  p_{m}(I) \, - \, p_{m}(\{ a_{1} - 1, a_{2} - 1, \ldots , a_{d} - 1 \}) \\
& &  \!\!\!\!\!\!\!\!\!\!\!\!\!\!\!\!\!\!\!\!\!\! \quad = \:
(-1)^{k+1} \sum_{i=0}^{l-1} \, \binom{k+i}{i} p_{m-k-1-i}(D)  \cdot  x_{m-k-i}
        \: + \: \sum_{j=0}^{k} \, (-1)^{j} \binom{l-1+j}{j} p_{m-l-j}(D_{j}) \cdot x_{m-l-j+1}\\
&  & \!\!\!\!\!\!\!\!\!\!\!\!\!\!\!\!\!\!\!\!\!\!\!\!\!\!\!\!\!\!\!
\, + \, (-1)^{k+1} \sum_{i=0}^{l} \, \binom{k-1+i}{i} p_{m-k-i}(D) \cdot  x_{m-k-i+1}
        \; + \; \sum_{j=0}^{k-1} \, (-1)^{j+1} \binom{l+j}{j} p_{m-l-1-j}(D_{j+1}) \cdot  x_{m-l-j}  \\
& &  \!\!\!\!\!\!\!\!\!\!\!\!\!\!\!\!\!\!\!\!\!\! \quad = \:
(-1)^{k+1} \big\{ \sum_{i=1}^{l} \, \binom{k+i-1}{i-1} p_{m-k-i}(D)  \cdot  x_{m-k-i+1}
\, + \, \sum_{i=0}^{l} \, \binom{k+i-1}{i} p_{m-k-i}(D)  \cdot  x_{m-k-i+1} \big\} \\
&  & \!\!\!\!\!\!\!\!\!\!\!\!\!\!\!\!\!\!\!\!\!\!\!\!\!\!\!\!\!\!\!
        \: + \: \big\{ \sum_{j=0}^{k} \, (-1)^{j} \binom{l-1+j}{j} p_{m-l-j}(D_{j}) \cdot x_{m-l-j+1}
        \; + \; \sum_{j=1}^{k} \, (-1)^{j} \binom{l-1+j}{j-1} p_{m-l-j}(D_{j}) \cdot x_{m-l-j+1} \big\} \\
& &  \!\!\!\!\!\!\!\!\!\!\!\!\!\!\!\!\!\!\!\!\!\! \quad = \:
(-1)^{k+1} \big\{ \sum_{i=1}^{l} \, \big[ \binom{k+i-1}{i-1} + \binom{k+i-1}{i} \big] p_{m-k-i}(D) \cdot x_{m-k-i+1} \, + \, p_{m-k}(D) \cdot  x_{m-k+1} \big\} \\
&  & \!\!\!\!\!\!\!\!\!\!\!\!\!\!\!\!\!\!\!\!\!\!\!\!\!\!\!\!\!\!\!
\: + \: \big\{ \sum_{j=1}^{k} \, (-1)^{j} \big[ \binom{l-1+j}{j} + \binom{l-1+j}{j-1} \big] p_{m-l-j}(D_{j}) \cdot x_{m-l-j+1}
\, + \, p_{m-l}(D_{0}) \cdot  x_{m-l+1} \big\} \\
& &  \!\!\!\!\!\!\!\!\!\!\!\!\!\!\!\!\!\!\!\!\!\! \quad = \:
(-1)^{k+1} \big\{ \sum_{i=1}^{l} \, \binom{k+i}{i} p_{m-k-i}(D) \cdot x_{m-k-i+1}\, + \, \binom{k + 0}{0} p_{m-k}(D) \cdot  x_{m-k+1} \big\} \\
&  & \!\!\!\!\!\!\!\!\!\!\!\!\!\!\!\!\!\!\!\!\!\!\!\!\!\!\!\!\!\!\!
\: + \: \big\{ \sum_{j=1}^{k} \, (-1)^{j} \binom{l+j}{j} p_{m-l-j}(D_{j}) \cdot x_{m-l-j+1}
\, + \, (-1)^{0}\binom{l+0}{0} p_{m-l}(D_{0}) \cdot  x_{m-l+1} \big\} \\
& &  \!\!\!\!\!\!\!\!\!\!\!\!\!\!\!\!\!\!\!\!\!\! \quad = \:
(-1)^{k+1} \sum_{i=0}^{l} \, \binom{k+i}{i} p_{m-k-i}(D) \cdot x_{m-k-i+1} \: + \:
\sum_{j=0}^{k} \, (-1)^{j} \binom{l+j}{j} p_{m-l-j}(D_{j}) \cdot x_{m-l-j+1},
\end{eqnarray*}
as required.
\end{proof}

\smallskip

By way of example let us calculate $p_{6}(I)$ for the subset $I = \{2,3,5\}$ of $[6]$. We write $p_{6}(2,3,5)$
for brevity and by Proposition 5.4 we obtain
\begin{equation*}
p_{6}(2,3,5) \: = \: {1 + 0 \choose 0} \, p_{4}(1,3) \cdot x_{5} \, + \, {1 + 1 \choose 1} \, p_{3}(1,3) \cdot x_{4} \, + \,
                   {1 + 0 \choose 0} \, p_{4}(2,3) \cdot x_{5} \, - \, {1 + 1 \choose 1} \, p_{3}(1,2) \cdot x_{4} \, .
\end{equation*}

Clearly $p_{4}(2,3) = 0$ since the subset $\{2,3\}$ of $[4]$ corresponds to the even palindrome $ab^{2}a$.
For $p_{3}(1,2)$ we obtain $p_{3}(1,2) = p_{2}(1,2) - p_{2}(1) \cdot x_3$. Again $\{1,2\} = [2]$ corresponds to
the word $b^{2}$ so we have $p_{2}(1,2) = 0$, hence $p_{3}(1,2) = - (-1)^{1-1}{2-1 \choose 1-1} \, x_{3} = - x_{3}$.
For the remaining terms we have $p_{3}(1,3) = p_{2}(1) \cdot x_{3} - p_{2}(2) \cdot x_{3} =
(-1)^{1-1}{2-1 \choose 1-1} \, x_{3} - (-1)^{2-1}{2-1 \choose 2-1} \, x_{3} = 2 x_{3}$ and consequently
$p_{4}(1,3) =  p_{3}(1,3) - p_{3}(2) \cdot x_{4} = 2 x_{3} - (-1)^{2-1}{3-1 \choose 2-1} \, x_{4} = 2 x_{3} + 2 x_{4}$.
Summing up we finally obtain
\[ p_{6}(2,3,5) \: = \: 1 \cdot ( 2 x_{3} + 2 x_{4} ) \cdot x_{5} \, +  \, 2 \cdot ( 2 x_{3} ) \cdot x_{4}  \, - \,
2 \cdot ( - x_{3} ) \cdot x_{4} \: = \: 6 x_3 x_4 \, + \, 2 x_3 x_5 \, + \, 2 x_4 x_5 . \]
Observe that the greatest common divisor of the coefficients of $p_{6}(2,3,5)$ is equal to $2$,
hence the word $ab^{2}aba$ of length $6$, that corresponds to $I$, has $c(w) = 2$
and therefore does not lie in the support of the free Lie algebra over a field of characteristic $2$.

\bigskip

By Theorem \ref{suppm|c(w)}\,(i) and the equivalence \eqref{lnequivln*} established by Theorem \ref{wln*ln{w}},
if we consider words on a two-lettered alphabet Problem \ref{suppfreeLieZm} will be equivalent to the following.

\begin{problem}\label{pnmodm}
Let $m$ be a non-negative integer with $m \neq 1$. Find all set partitions $J \, / \, I$ of $[n]$
with the property that
\begin{equation*}
p_{n}(I) \: \equiv  \: 0 \: (\bmod \, m ) \, .
\end{equation*}
\end{problem}

For $m = 0$ Proposition \ref{kerlnZ} implies that the only solutions to Problem \ref{pnmodm} are - except from the trivial
solution where $I = \emptyset$, for each $n$ - all the subsets $I$ of $[n]$ fixed
by the involution ${\tau}_{n}$ defined by (\ref{invo}), when $n$ is even.

\smallskip

For $m > 1$ our next result provides a necessary condition - stated explicitly on $n, m$ and $I$ - for
$p_{n}(I) \, \equiv  \, 0 \, (\bmod m )$ to hold.

\smallskip

\begin{proposition}\label{Nn}
Let $m$ and $n$ be a positive integers with $m > 1$, $I$ be a subset of $[n]$ with and
$N_{n}(I)$ be the integer defined as
\[ N_{n}(I) \: = \: \sum_{i \in I} (-1)^{i-1} {{n-1} \choose {i-1}}. \]
\begin{description}
\item[\em{(i)}]
If $p_{n}(I) \: \equiv  \: 0 \: (\bmod m )$ then $m \mid N_{n}(I)$.
\item[\em{(ii)}]
In particular, if $m = p$, a prime number, $n = p^{e}$ and $p \nmid |I|$ then $p_{n}(I) \: \not\equiv  \: 0 \: (\bmod m )$.
\end{description}
\end{proposition}

\begin{proof}
(i) We need to relate the polynomial $h_{n}(I) \in {\mathbb Z}[x_{1}, x_{2}, \ldots , x_{n}]$ with the integer
$N_{n}(I)$. Let $|I| = s$. We claim that if we set $x_{1} = 1$ and $x_{2} = x_{3} = \cdots = x_{n} = t$ in $h_{n}(I)$ we obtain
the polynomial specialization
\begin{equation}
h_{n}(I)(1, t, t, \ldots , t) \: = \: N_{n}(I)\,t^{s-1} - N_{n}(I)\,t^{s}. \label{hnNn}
\end{equation}
If \eqref{hnNn} is true then it is evident that $m \mid N_{n}(I)$. Since
$h_{n}(I) = (x_{1} - x_{2}) \, p_{n}(I)$ by Proposition 5.3, equation \eqref{hnNn} yields
\begin{equation}
p_{n}(I)(1, t, t, \ldots , t) \: = \: N_{n}(I)\,t^{s-1}. \label{pnNn}
\end{equation}
Thus $N_{n}(I)$ may be given a combinatorial interpretation as the sum of the coefficients appearing in all monomials of $p_{n}(I)$.

To prove our claim we use the additive formula \eqref{addln} for $l_n$.
Since $h_{n}(I) = \psi ( l_n \cdot \overline{I} )$ we obtain
\begin{equation}
 h_{n}(I) \: = \: \sum_{k=1}^{n} \, (-1)^{k-1} \psi (D_{[k-1]} \cdot \overline{I} \, ).  \label{sumpsidesc}
\end{equation}
A typical element of $D_{[k-1]}$ is a permutation ${\pi}$ which, when viewed as word in $n$ distinct letters
from $[n]$, is written as
\[ {\pi} \: = \: j_{1} \, j_{2} \, \ldots \, j_{k-1} \, \fbox{$j_{k}$} \, j_{k+1} \, \ldots \, j_{n-1} \, j_{n} \, ,\]
where $j_{1} > j_{2} > \cdots > j_{k-1} > j_{k} = 1 < j_{k+1} < \cdots < j_{n-1} < j_{n}$. It is then clear that
$\displaystyle |D_{[k-1]}| = {{n-1} \choose {k-1}}$ and
$\displaystyle {\psi}({\pi} \cdot \overline{I}) = x_{j_{i_{1}}} x_{j_{i_{2}}} \cdots x_{j_{i_{s}}}$,
when $I = \{i_{1}, i_{2}, \ldots , i_{s} \}$. Setting $x_{1} = 1$ and $x_{2} = x_{3} = \cdots x_{n} = t$ we obtain
\begin{equation}
 \psi ( D_{[k-1]} \cdot \overline{I} )(1, t, t, \ldots , t) = \begin{cases}
   \displaystyle {{n  - 1 } \choose {k - 1}} \, t^{s-1},  &\text{if \, $k \in I$;} \\
   \\
   \displaystyle {{n  - 1 } \choose {k - 1}} \, t^{s},   &\text{otherwise.}
               \end{cases}    \label{psidesc}
\end{equation}
Summing up all these elements by \eqref{sumpsidesc} we obtain
\[ h_{n}(I)(1, t, t, \ldots , t) = \sum_{i \in I} (-1)^{i-1} {{n-1} \choose {i-1}}t^{s-1} \, + \,
\sum_{i \notin I} (-1)^{i-1} {{n-1} \choose {i-1}}t^{s} \]
and our claim follows since
$\displaystyle \sum_{i \in I} (-1)^{i-1} {{n-1} \choose {i-1}} \, + \, \sum_{i \notin I} (-1)^{i-1} {{n-1} \choose {i-1}} =
\sum_{i = 1}^{n} (-1)^{i-1} {{n-1} \choose {i-1}} = 0$, by the binomial theorem.

(ii) We claim that $N_{n}(I) = s$. If this is true the result will follow immediately by part (i) and the assumption
$p \nmid s$. To prove the claim it suffices to show that for each $k \in \{ 0, \ldots , p^{e} - 1 \}$
\begin{equation}
(-1)^{k}{ {p^{e} - 1} \choose k} \equiv 1  \, (\bmod p ). \label{p^e}
\end{equation}
We apply Lucas correspondence theorem and write the numbers $p^{e} - 1$ and $k$ in base $p$ respectively as 
$p^{e} - 1 =  \sum_{t=0}^{e-1} (p-1) \, p^{t}$ and $k =  \sum_{t=0}^{e-1} k_{t} \, p^{t}$, where $k_{t} \in \{ 0, 1, \ldots , p-1 \}$,
so we need to show that 
$\displaystyle (-1)^{l}{ {p - 1} \choose l} \equiv 1  \, (\bmod p )$, for each $l \in \{ 0, \ldots , p - 1 \}$.
The latter follows from the fact that $\displaystyle p  \mid \! {p \choose l} = {{p - 1} \choose l} + {{p - 1} \choose {l - 1}}$, which yields
$\displaystyle {{p - 1} \choose l} \equiv - {{p - 1} \choose {l - 1}} \, (\bmod p )$, and the rest is done by a straightforward induction on $l$.
\end{proof}

\medskip

\begin{corollary}\label{suffsupp}
Let $u,v$ and $w$ be words of length $n$ with $alph(w) = \{a, b\}$; $I, J$ and $K$ be the subsets of $[n]$
consisting of the positions that $b$ occurs in $u,v$ and $w$, respectively and $N_{n}(I)$, $N_{n}(J)$ and $N_{n}(K)$ be defined as in
Proposition \ref{Nn}. If $m$ is a non-negative integer with $m \neq 1$ then
\begin{description}
\item[\em{(i)}]
If $m \nmid N_{n}(K)$ the word $w$ lies in the support of the free Lie algebra ${\mathcal L}_{{\mathbb Z}_{m}}(A)$.
In particular, if $m = p$, a prime number and $n = p^{e}$ every word $w$ with $|K| = s$ and $p \nmid s$ lies
in the support of ${\mathcal L}_{{\mathbb Z}_{m}}(A)$.
\item[\em{(ii)}]
If $u, v$ is a twin pair of words with respect to ${\mathcal L}_{{\mathbb Z}_{m}}(A)$ then
$N_{n}(I) \equiv N_{n}(J) \, (\bmod \, m)$.
\item[\em{(iii)}]
If $u, v$ is an anti-twin pair of words with respect to ${\mathcal L}_{{\mathbb Z}_{m}}(A)$ then
$N_{n}(I) \equiv - N_{n}(J) \, (\bmod \, m)$.
\end{description}
\end{corollary}

\begin{proof}
Part (i) is a direct consequence of Proposition \ref{Nn}.
For parts (ii) and (iii) we use Theorem \ref{suppm|c(w)}\,(ii). For a twin pair we obtain
$l^{*}(u) \, - \, l^{*}(v) \in (m)\,{\mathbb Z} \langle A \rangle$ which yields
$p_{n}(I) \, - \, p_{n}(J) \equiv 0 \, (\bmod \, m)$ by Theorem \ref{wln*ln{w}} and our result follows by \eqref{pnNn}.
For an anti-twin pair we argue similarly for the polynomial $l^{*}(u) \, + \, l^{*}(v)$ and its commutative analogue
$p_{n}(I) \, + \, p_{n}(J)$.
\end{proof}

\smallskip

{\em Remarks.}
The invariant $N_{n}(I)$ is a partial sum of signed binomial coefficients from the $n$-th row
of the Pascal triangle (starting to count from $n = 0$). Considering this row $\bmod \, m$ Corollary \ref{suffsupp}\,(i) poses
the requirement that the signed sum of the entries in the positions appearing in $I$ has to be different from zero $\bmod \, m$,
for the word $w$ corresponding to $I$ to lie in the support of ${\mathcal L}_{{\mathbb Z}_{m}}(A)$.
In the binary case this simply means that the digit $1$ is allowed to appear an odd number of times in these positions. Note that the number
of appearances of $1$'s in each row of the Pascal triangle $\bmod \, 2$ is always a power of $2$; it is in fact equal to $2^{b(n)}$, where $b(n)$ is the number of occurrences of the digit $1$ in the binary representation of $n$ (\textit{e.g.}, see \cite{Fine} or \cite{Wolf}).

The condition of Corollary \ref{suffsupp}\,(i) is sufficient but not necessary since there exist many words $w$ with
$m \mid N_{n}(I)$ that also lie in the support of the free Lie algebra ${\mathcal L}_{{\mathbb Z}_{m}}(A)$.
For example, for $m = 2$ and $w = a^{2}b^{2}a$ we get $I = I(w) = \{2,3\}$ and $N_{5}(I) = 2$, but $w$ lies
in the support of ${\mathcal L}_{{\mathbb Z}_{2}}(A)$ since $p_{5}(I) = x_{3} + x_{4}$ and $c(w) = 1$.
On the other hand, for $|I| = 1$ it is a necessary and sufficient condition identified with our theoretical characterization of the
support of ${\mathcal L}_{{\mathbb Z}_{m}}(A)$ and is checked with Kummer's Lemma.
Note also that it is even possible to have $N_{n}(I) = 0$ with $p_{n}(I) \neq 0$, \textit{e.g.}, 
for $w = a^{2}b^{2}a^{2}ba^{2}$ we get $I = \{ 3, 4, 7 \}$ and $N_{9}(I) = 0$. 
Similar examples which demonstrate that the converse of Corollary \ref{suffsupp}\,(ii) and (iii) does not hold can be found, \textit{e.g.},
consider the words $u = ab^{3}a^{2}ba^{2}$ and $v = a^{2}b^{2}a^{2}b^{2}a$ of length $9$ with corresponding $I = I(u) = \{ 2, 3, 4, 7 \}$ and
$J = J(v) = \{ 3, 4, 7, 8 \}$. Then clearly $N_{9}(I) = N_{9}(J) = -8$ but $p_{9}(I) \neq p_{9}(J)$, as one can check by
Theorem \ref{recformulapn}, so that $l^{*}(u) \neq l^{*}(v)$ and hence $u, v$ is not a pair of twin words.

\smallskip

Finally, in view of Reduction Theorem \ref{reducetoalphof2}, Conjecture \ref{conjtabloidZ} for tabloids boils down to the following one
for Pascal descent polynomials.

\begin{conjecture}\label{conjpn}
Let $I, \, J$ be subsets of $[n]$ of cardinality $s \leq {\lfloor n/2 \rfloor}$ with $p_{n}(I) \neq 0$
and $p_{n}(J) \neq 0$. Then
\begin{description}
 \item[\em{(i)}]
$p_{n}(I) = p_{n}(J)$ if and only if $I = J$ or $n$ is odd and $I = {{\tau}_{n}}(J)$.
 \item[\em{(ii)}]
$p_{n}(I) = - p_{n}(J)$ if and only if $n$ is even and $I = {{\tau}_{n}}(J)$.
\end{description}
\end{conjecture}

For $s = 1$, \textit{i.e.}, when $I = \{i\}$ and $J = \{j\}$, Conjecture \ref{conjpn} holds as a special case (for $r = n - 1$, $k = i - 1$ and $l = j - 1$)
of the fact that $\displaystyle \binom{r}{k} = \binom{r}{l}$ if and only if
$k = l$ or $k = r - l$. The "if" part, known as the {\em symmetry identity}, follows directly from the definition of binomial coefficients
and the "only if" part follows from the inequality $\displaystyle \binom{r}{k} <  \binom{r}{k+1}$ when $1 \leq k + 1 \leq \lfloor r/2 \rfloor$.
This is another indication showing that the Pascal descent polynomial $p_{n}(I)$ is indeed an extension of the usual notion
of the binomial coefficient.

\bigskip


\section{ Further Research}\label{futur}

Various equivalent forms of Conjecture \ref{conjtwinantiZ} on twin and anti-twin words - which is enough to prove on a
two-lettered alphabet - have been presented in this article and we strongly  believe that this will finally be resolved.
More precisely one can use the recursive formula of Proposition \ref{recl*shuffle} in a manner similar to the proofs of Theorem \ref{DuchThib}
and Theorem \ref{alphbound}, but as it turns out, a lot more cases in combinatorics on words appear in such a consideration.

By Remark \ref{propershuffle} and Reduction Theorem \ref{reducetoalphof2}, another equivalent conjecture which seems to be worth investigating 
is the following.
Suppose that $u$ and $v$ are words of common multi-degree and length $n > 1$ on a two-lettered alphabet which are not powers of
a single letter or palindromes of even length. Then the binomial $u - v$ (respectively $u + v$) can be expressed as a $K$-linear
combination of proper shuffles if and only if either $u = v$ or $n$ is odd and $u = \tilde{v}$
(respectively if $n$ is even and $u = \tilde{v}$). One is challenged to check this using one of the two well known bases of the shuffle algebra,
namely the triangular $\mathbb Z$-basis - originally due to Radford \cite{Radf} -
$\displaystyle {\mathcal Q}_{w} = \frac{1}{i_{1}! \ldots i_{k}!} l_{1}^{ \, {\sqcup \mspace{-3.0mu} \sqcup} i_{1}} \sh \cdots \sh \,
l_{k}^{ \, {\sqcup \mspace{-3.0mu} \sqcup} i_{k}}$, where $w = {l_{1}}^{i_{1}} \cdots {l_{k}}^{i_{k}}$ is the unique decreasing factorization
(with respect to the lexicographical order in $A^{+}$) of a non-Lyndon word $w$ as product of Lyndon words with
$l_{1} > \cdots > l_{k}; \, i_{1}, \ldots , i_{k} \geq 1$ (see \cite[\S 6.1]{Reut}) and the Lie polynomial $\mathbb Q$-basis
$\displaystyle {\bigoplus}_{n \geq 0} V_{n}$, where $V_{n}$ denotes the subspace of ${\mathbb Q}{\langle A \rangle}$ generated by the shuffle
products of $n$ Lie polynomials (see \cite[\S 6.5.1]{Reut}).

A completely different approach via Pascal descent polynomials is to use Theorem \ref{recformulapn} in order to be able to resolve
Conjecture \ref{conjpn}.
We do not yet know how or even if Theorem \ref{recformulapn} could lead to a complete solution of Problem \ref{pnmodm} (probably combined with
successive applications of our condition in Proposition \ref{Nn} involving the invariant $N_{n}(I)$) but we are certain that
the right framework for such a search is within the geometry of the Pascal triangle $\bmod \, m$,
which after all is needed even in the simple case where $|I| = 1$.

Finally, in the case where $m = 2$ - the smallest instance of Problem 5.5 - we have made some computations using the computer algebra
system \textsc{Gap 4} and have obtained all solutions up to $n=12$. (A list of those solutions up to $n =10$ is presented in the
following Appendix.)


\begin{appendix}
\section{Solutions of Problem \ref{pnmodm} for $\mathbf{m = 2}$ and $\mathbf{n \leq 10}$}

For $3 \leq n \leq 10$ and $1 \leq s \leq \lfloor n/2 \rfloor$ we list all subsets $I$ of $[n]$ of cardinality $s$ with the property that
$p_{n}(I) \, \equiv  \, 0 \, (\bmod \, 2 )$. By Corollary \ref{suffsupp}\,(i) there are no solutions 
when $n$ is a power of $2$ and $s$ is odd.

\begin{align*}
\bf{n}& \bf{= 3} & \bf{s}& \bf{ = 1} & : &  \quad \{2\} \\
      &          &       &           &   & \\
\bf{n}& \bf{= 4} & \bf{s}& \bf{ = 2} & : &  \quad \{1,3\}, \, \{1,4\}, \, \{2,3\}, \, \{2,4\} \\
      &          &       &           &   & \\
\bf{n}& \bf{= 5} & \bf{s}& \bf{ = 1} & : &  \quad \{2\}, \, \{3\}, \, \{4\} \\
      &          & \bf{s}& \bf{ = 2} & : &  \quad \{1,5\}, \, \{2,4\} \\
      &          &       &           &   & \\
\bf{n}& \bf{= 6} & \bf{s}& \bf{ = 1} & : &  \quad \{3\}, \, \{4\}  \\
      &          & \bf{s}& \bf{ = 2} & : &  \quad \{1,5\}, \, \{1,6\}, \, \{2,5\}, \, \{2,6\}, \, \{3,4\} \\
      &          & \bf{s}& \bf{ = 3} & : &  \quad \{1,3,5\}, \, \{1,3,6\}, \, \{1,4,6\}, \, \{2,3,5\}, \,
                                                  \{2,4,5\}, \, \{2,4,6\} \\
      &          &       &           &   & \\
\bf{n}& \bf{= 7} & \bf{s}& \bf{ = 1} & : &  \quad \{2\}, \, \{4\}, \, \{6\} \\
      &          & \bf{s}& \bf{ = 2} & : &  \quad \{1,5\}, \, \{1,7\}, \, \{2,4\}, \, \{2,6\}, \, \{3,5\}, \, \{3,7\}, \, \{4,6\} \\
      &          & \bf{s}& \bf{ = 3} & : &  \quad \{1,3,6\}, \, \{1,4,7\}, \, \{2,4,6\}, \, \{2,5,7\}, \, \{3,4,5\} \\
      &          &       &           &   & \\
\bf{n}& \bf{= 8} & \bf{s}& \bf{ = 2} & : &  \quad \{1,5\}, \, \{1,7\}, \, \{1,8\}, \, \{2,5\}, \, \{2,6\}, \, \{2,7\}, \, \{2,8\}, \,
                                                  \{3,5\}, \, \{3,6\}, \, \{3,7\}, \\
      &          &       &           &   &  \quad \{4,5\}, \, \{4,6\},  \, \{4,7\}, \, \{4,8\} \\
      &          & \bf{s}& \bf{ = 4} & : &  \quad \{1,2,7,8\}, \, \{1,3,5,7\}, \, \{1,3,5,8\}, \, \{1,3,6,8\}, \, \{1,4,5,8\}, \,
                                                  \{1,4,6,8\}, \, \{2,3,5,7\}, \\
      &          &       &           &   &  \quad \{2,3,6,7\}, \, \{2,4,5,7\}, \, \{2,4,6,7\}, \, \{2,4,6,8\}, \,                                                  \{3,4,5,6\} \\
      &          &       &           &   & \\
\bf{n}& \bf{= 9} & \bf{s}& \bf{ = 1} & : &  \quad \{2\}, \, \{3\}, \, \{4\}, \, \{5\}, \, \{6\}, \, \{7\}, \, \{8\} \\
      &          & \bf{s}& \bf{ = 2} & : &  \quad \{1,9\}, \, \{2,4\}, \, \{2,6\}, \, \{2,8\}, \, \{3,6\}, \, \{3,7\}, \, \{4,6\}, \,
                                                  \{4,7\}, \, \{4,8\}, \, \{6,8\} \\
      &          & \bf{s}& \bf{ = 3} & : &  \quad \{1,5,9\}, \, \{2,4,6\}, \, \{2,4,8\}, \, \{2,5,7\}, \, \{2,5,8\}, \,
                                                  \{2,6,8\}, \, \{3,5,7\}, \, \{3,5,8\}, \\
      &          &       &           &   &  \quad \{4,5,6\}, \, \{4,6,8\} \\
      &          & \bf{s}& \bf{ = 4} & : &  \quad \{1,2,8,9\}, \, \{1,3,5,9\}, \, \{1,3,6,9\}, \, \{1,3,7,9\}, \, \{1,4,6,9\}, \,
                                                  \{1,4,7,9\}, \, \{1,5,7,9\}, \\
      &          &       &           &   &  \quad \{2,3,4,8\}, \, \{2,3,7,8\}, \, \{2,4,6,8\}, \, \{2,6,7,8\}, \, \{3,4,6,7\} \\
      &          &       &           &   &  \\
\bf{n}& \bf{= 10} & \bf{s}& \bf{ = 1} & : &  \quad \{3\}, \, \{4\}, \, \{5\}, \, \{6\}, \, \{7\}, \, \{8\} \\
      &           & \bf{s}& \bf{ = 2} & : &  \quad \{1,9\}, \, \{1,10\}, \, \{2,9\}, \, \{2,10\}, \, \{3,7\}, \, \{3,8\}, \, \{4,7\}, \,
                                                  \{4,8\}, \, \{5,6\} \\
      &           & \bf{s}& \bf{ = 3} & : & \quad \{1,5,9\}, \, \{2,5,9\}, \, \{2,6,9\}, \, \{2,6,10\}, \, \{3,5,7\}, \,
                                                  \{4,5,7\}, \, \{4,6,7\}, \, \{4,6,8\} \\
      &           & \bf{s}& \bf{ = 4} & : & \quad \{1,2,9,10\}, \, \{1,3,5,9\}, \, \{1,3,6,9\}, \, \{1,3,6,10\}, \, \{1,3,7,9\}, \,
                                                  \{1,3,8,10\}, \\
      &           &       &           &   & \quad \{1,4,6,9\}, \, \{1,4,6,10\}, \, \{1,4,7,10\}, \, \{1,5,6,10\}, \, \{1,5,7,9\}, \,
                                                  \{1,5,7,10\},\\
      &          &      &        &   & \quad \{1,5,8,10\}, \, \{2,3,5,9\}, \, \{2,3,7,9\}, \, \{2,3,8,9\}, \, \{2,4,5,9\}, \, \{2,4,6,9\}, \\                                                  &          &      &        &   & \quad \{2,4,6,10\}, \, \{2,4,7,9\}, \, \{2,4,8,9\}, \, \{2,4,8,10\}, \, \{2,5,6,9\}, \, \{2,5,7,9\},\\      &          &      &        &   & \quad \{2,5,7,10\}, \, \{2,5,8,10\}, \, \{2,6,7,9\}, \, \{2,6,8,9\}, \, \{2,6,8,10\}, \, \{3,4,7,8\},\\
      &          &      &        &   & \quad \{3,5,6,8\}, \, \{4,5,6,7\} \\
      &           & \bf{s}& \bf{ = 5} & : & \quad \{1,3,4,5,9\}, \, \{1,3,4,6,10\}, \, \{1,3,5,7,9\}, \, \{1,3,5,7,10\}, \, \{1,3,5,8,10\}, \\
      &          &      &        &   & \quad \{1,3,6,8,10\}, \, \{1,4,6,8,10\}, \, \{1,5,7,8,10\}, \, \{2,3,4,6,9\}, \, \{2,3,5,7,9\}, \\
      &          &      &        &   & \quad \{2,4,5,7,9\}, \, \{2,4,6,7,9\}, \, \{2,4,6,8,9\}, \, \{2,4,6,8,10\}, \, \{2,5,7,8,9\}, \\
      &          &      &        &   & \quad \{2,6,7,8,10\}.
\end{align*}
\end{appendix}

{\bf Acknowledgements}

\smallskip

The author wishes to thank professor G. Duchamp for helpful discussions on this subject - in
particular for pointing out Problems 1.3 and 1.4 on twin and anti-twin words - during the author's visit at \textsc{Lipn}
(Laboratoire d' Informatique de Paris-Nord) at the University of Paris XIII.
He is also indebted to professor A. Konovalov for his advice on computations using \textsc{Gap 4} to implement
the action of $l_n$ on subsets of $[n]$.

\bigskip
\bigskip



\begin{thebibliography}{99}
\bibitem{Diek + Rozen} V. Diekert and G. Rozenberg (editors), \textsl{The Book of Traces}, World Scientific, 1995.

\bibitem{Duch + Laug + Luqu} G. Duchamp, \'{E}. Laugerotte, J.-G. Luque,
\textsl{On the support of graph Lie algebras}, Theoret. Comput. Sci. 273 (2002) 283-294.

\bibitem{Duch + Thib} G. Duchamp, J.-Y. Thibon, \textsl{Le support de l'alg\`{e}bre de Lie libre},
Discrete Math. 76 (1989) 123-132.

\bibitem{Fine} N. J. Fine, \textsl{Binomial coefficients modulo a prime}, The American Mathematical Monthly, 54, no. 10, part 1, (1984) 589-592.

\bibitem{Grah + Knut + Pata} R. Graham, D. Knuth, O. Patashnik, \textsl{Concrete Mathematics : a foundation for
computer science}, 2nd ed., Addison-Wesley, 1994.

\bibitem{Kumm} E. E. Kummer, \textsl{\"Uber die Erg\"{a}nzungss\"{a}tze den allgemeinen Reciprocit\"{a}tsgesetzen},
J. Reine Angew. Math., 44 (1852) 93-146. Reprinted in his Collected Papers, Vol. 1, 485-538.

\bibitem{Loth} M. Lothaire, \textsl{Combinatorics on Words}, Encyclopedia of Mathematics and its Applications,
Vol. 17, Addison-Wesley, Reading, 1983.

\bibitem{Luca} E. Lucas, \textsl{Sur les congruences des nombres eul\'{e}riens et des coefficients differentiels des fonctions
trigonom\'{e}triques, suivant un module premier}, Bulletin de la Soci\'{e}t\'{e} math\'{e}matique de France, 6 (1877) 49-54.

\bibitem{Radf} D. E. Radford, \textsl{A natural ring basis for the shuffle algebra and an application to group schemes},
Jour. of Algebra 58 (1979) 432-454.

\bibitem{Ree} R. Ree, \textsl{Lie elements and algebra associated with shuffles}, Ann. Math. 68 (1958) 210-220.

\bibitem{Reut} C. Reutenauer, \textsl{Free Lie Algebras}, London Mathematical Society New Series, Vol. 7,
Oxford University Press, London, 1993.

\bibitem{Saga} B. E. Sagan, \textsl{The Symmetric Group: Representations, Combinatorial Algorithms,
and Symmetric Functions}, Second Edition, Graduate Texts in Mathematics, Vol. 203,
Springer-Verlag, New York, 2001.

\bibitem{Scho} M. Schocker, \textsl{The descent algebra of the symmetric group},
Representations of finite dimensional algebras and related topics in Lie theory and geometry, Fields Inst. Comm. 40 (2004) 145-161.

\bibitem{Wolf} S. Wolfram, \textsl{Geometry of binomial coefficients}, The American Mathematical Monthly, 91, no. 9, (1984) 566-571.
\end{thebibliography}
\end{document}